\documentclass[11pt,letterpaper]{amsart}

\usepackage[margin=1in]{geometry}
\usepackage{amsmath, bm, bbm, amssymb, amsfonts, mathtools, eucal}
\usepackage{graphicx}
\usepackage{setspace}
\usepackage[shortlabels]{enumitem}
\usepackage{subcaption}
\usepackage{multicol}
\usepackage[normalem]{ulem} % cancel text

\newtheorem{theorem}{Theorem}[section]
\newtheorem{proposition}[theorem]{Proposition}
\newtheorem{lemma}[theorem]{Lemma}
\newtheorem{corollary}[theorem]{Corollary}

\theoremstyle{definition}
\newtheorem{definition}[theorem]{Definition}
\newtheorem{question}{Question}
\newtheorem{example}[theorem]{Example}
\newtheorem{open}{Open Problem}

\theoremstyle{remark}

\usepackage{hyperref,xcolor}
\hypersetup{
colorlinks,
linkcolor={red!50!black},
citecolor={blue!50!black},
urlcolor={blue!80!black}
}

\DeclareMathOperator*{\argmin}{argmin}

\DeclareMathOperator{\tv}{\scriptscriptstyle TV}
\DeclareMathOperator{\kl}{\scriptscriptstyle KL}
\DeclareMathOperator{\js}{\scriptscriptstyle JS}
\DeclareMathOperator{\hl}{\scriptscriptstyle H}
\DeclareMathOperator{\lc}{\scriptscriptstyle LC}
\DeclareMathOperator{\D}{D}
\DeclareMathOperator{\proj}{proj}

\DeclareMathOperator{\diag}{diag}
\DeclareMathOperator{\vc}{\scriptscriptstyle VC}

\newcommand{\p}{{\scriptscriptstyle+}}
\newcommand{\tp}{{\scriptscriptstyle\mathsf{T}}}
\newcommand{\1}{\mathbbm{1}}
\renewcommand{\d}{\mathrm{d}} 
\newcommand{\Cint}{\mathsf{C}\!\!\!\!\!\! \int}

\let\latexcirc=\circ
\newcommand{\ccirc}{\mathbin{\mathchoice
  {\xcirc\scriptstyle}
  {\xcirc\scriptstyle}
  {\xcirc\scriptscriptstyle}
  {\xcirc\scriptscriptstyle}
}}
\newcommand{\xcirc}[1]{\vcenter{\hbox{$#1\latexcirc$}}}
\let\circ\ccirc

\begin{document}
\title{Glivenko--Cantelli for $f$-Divergence}
\author[H.~Wang]{Haoming~Wang}
\address{Computational and Applied Mathematics Initiative, Department of Statistics,
University of Chicago, Chicago, IL 60637}
\email{haomingwang@uchicago.edu}
\author[L.-H.~Lim]{Lek-Heng~Lim}
\address{Computational and Applied Mathematics Initiative, Department of Statistics,
University of Chicago, Chicago, IL 60637}
\email{lekheng@uchicago.edu}

\begin{abstract}
We extend the celebrated Glivenko--Cantelli theorem, sometimes called the fundamental theorem of statistics, from its standard setting of total variation distance to all $f$-divergences. A key obstacle in this endeavor is to define $f$-divergence on a subcollection of a $\sigma$-algebra that forms a $\pi$-system but not a $\sigma$-subalgebra. This is a side contribution of our work. We will show that this notion of $f$-divergence on the $\pi$-system of rays preserves nearly all known properties of standard $f$-divergence, yields a novel integral representation of the Kolmogorov--Smirnov distance, and has a Glivenko--Cantelli theorem. We will also discuss the prospects of a Vapnik--Chervonenkis theory for $f$-divergence.
\end{abstract}


\maketitle

\section{Introduction}\label{sec:intro}

The Glivenko--Cantelli theorem \cite{glivenko1933, cantelli1933} is a cornerstone of empirical process theory. It is likely the best known statement regarding the asymptotic behavior of stochastic processes formed by empirical measures \cite{Dudley2024}. It is also crucial in nonparametric statistics and forms the basis for statistical consistency in many estimation problems \cite{Dudley2024}. In the Kolmogorov--Smirnov test, the theorem ensures that the test statistic has desirable asymptotic properties \cite{Shorack2009}. In statistical resampling methods like the bootstrap, it guarantees that the empirical distribution derived from resampled data will approximate the true distribution as the number of samples grows \cite{Efron1993}. Because of its many central roles, the Glivenko--Cantelli theorem is often called the fundamental theorem of statistics \cite{Taylor1997, Henze2024, Dehling2011, Schumacher2017}. 

In machine learning, the Glivenko--Cantelli theorem can be extended to the Vapnik--Chervonenkis theorem \cite{VapnikChervonenkis1971}, also known as the fundamental theorem of learning theory \cite{shalev2014}. It is used to show the consistency of the principle of empirical risk minimization (ERM) \cite{Vapnik1998} and to derive bounds on generalization error by ensuring the difference between empirical and true distributions diminishes with larger samples, aiding in over-fitting control \cite{Bousquet2002}.

Given its pivotal role in both statistics and machine learning, it is surprising that the Glivenko--Cantelli theorem is somewhat limited in scope---it only works with the \emph{total variation norm} restricted to a Glivenko--Cantelli class. Modern AI applications, on the other hand, have ushered in a plethora of alternative distances between two probability distributions $\mu$ and $\nu$, most of them easier to compute than the total variation norm. The most prominent of these are the $f$-divergences
\begin{equation}\label{eq:f}
\D_f(\mu \Vert \nu) = \int_\Omega f\Bigl(\frac{\d\mu}{\d\nu}\Bigr)\, \d\nu = \int_\Omega f\bigl(g(x)\bigr)\, \d\nu(x),
\end{equation}
where $g$ is the Radon--Nikodym derivative with $\d\mu(x)= g(x)\, \d\nu(x)$. The total variation norm is itself an $f$-divergence with  $f(t) = \lvert t - 1 \rvert/2$  but many other $f$-divergence with yet other $f$'s have played a prominent role in important AI applications recently. In \cite{KL-kingma2014}, the Kullback--Leibler (KL) divergence is used to regularize the posterior distribution in the Variational Autoencoder (VAE), aligning it with a chosen prior. In \cite{JS-goodfellow2014}, Generative Adversarial Networks (GANs) are trained by minimizing the Jensen--Shannon (JS) divergence between the model and real data distributions. In \cite{Renyi-li2016}, R\'{e}nyi $\alpha$-divergences provide a smooth interpolation from the evidence lower bound to the log marginal likelihood that is controlled by the parameter $\alpha$. The goal of our article is to generalize the Glivenko--Cantelli theorem to any $f$-divergence, restricted to a Glivenko--Cantelli class defined with respect to this $f$-divergence.

At this point, one may wonder about Wasserstein metrics, another class of distances between probability distributions that rivals $f$-divergences in importance and popularity. But as was shown in \cite{bolley2007, liu2019}, there is no equivalent of Glivenko--Cantelli theorem  for $p$-Wasserstein metric. Such a result will always be trivial as the ``Glivenko--Cantelli class'' with respect to any $p$-Wasserstein metric is just the entire Borel $\sigma$-algebra: see \cite{bolley2007} for $p \in [1,\infty)$ and \cite{liu2019} for $p = \infty$. On the other hand,  the main contribution of our article shows that there is a natural, nontrivial generalization of the Glivenko--Cantelli theorem to $f$-divergences of which  the special case $f(t) = \lvert t - 1 \rvert/2$ gives us the classical Glivenko--Cantelli theorem.

We remind the reader of this classical version, taking the opportunity to set notations, define basic notions, and establish some facts for later use. Let $(\Omega,\Sigma,P)$ be a probability space and $X : \Omega \to \mathbb{R}$ be a random variable. Let  $\nu$ be the distribution of $X$, i.e., the pushforward measure of $P$ by $X$. Thus $\nu$ is a probability measure on the real Borel space $(\mathbb{R},\mathcal{B})$. For $X_1, \dots, X_n\sim \nu$ independent and identically distributed  random variables on the probability space $(\mathbb{R},\mathcal{B},\nu)$, the \emph{empirical distribution} is the random measure
\[
\nu_n(A,\omega) \coloneqq \frac{1}{n}\sum_{i=1}^{n} \1_{\{X_i(\omega)\in A\}},
\]
defined for any $A\in\mathcal{B}$ and $\omega\in\Omega$. For a fixed $\omega\in\Omega$, $\nu_n(\cdot,\omega)$ is a probability measure on $(\mathbb{R},\mathcal{B})$; and for a fixed $A\in\mathcal{B}$, $\nu_n(A) \coloneqq \nu_n(A,\cdot)$ is a random variable on $\Omega$. 

As is well known, the strong law of large number states that as $n \to \infty$,
\begin{equation}
\tag*{SLLN:} \text{for each fixed }A, \; \nu_n(A) - \nu(A) \to 0 \text{ almost surely.}
\end{equation}
The Glivenko--Cantelli theorem may be viewed as an attempt to replace the qualifier  ``for each fixed $A$''  by ``for all $A \in \mathcal{C}$'' where  $\mathcal{C}$ is an appropriately chosen class of subsets of $\Omega$ called a Glivenko--Cantelli class \cite{Shorack2009}, while maintaining the property that as $n \to \infty$,
\begin{equation}
\tag*{GC:} \sup_{ A \in \mathcal{C}} \, \bigl( \nu_n(A) - \nu(A) \bigr) \to 0  \text{ almost surely.}
\end{equation}
The following example shows that one cannot simply take $\mathcal{C} = \mathcal{B}$, the Borel $\sigma$-algebra.
\begin{example}\label{ex:tv}
Let $\nu$ be the standard Gaussian distribution. Let $X_1, \dots, X_n\sim \nu$ be independent and identically distributed random variables with realizations $x_1,\dots,x_n$. If $A=\{x_1,\dots, x_n\}\in\mathcal{B}$, then $\nu_n(A)=1$ but $\nu(A)=0$. Thus $\sup_{A\in\mathcal{B}}\bigl(\nu_n(A)-\nu(A)\bigr) = 1$ for all $n$ and $\mathcal{B}$ is not a Glivenko--Cantelli class.
\end{example}
One might think that $\mathcal{B}$ is too big but even if we choose a much smaller $\mathcal{C}$, it could still fail to be a Glivenko--Cantelli class, as the following example from \cite{Shortt1984} shows.
\begin{example}[Shortt]\label{ex:gc}
Let $\nu$ be the uniform measure on $[0,1]$ and $\mathcal{C}$ be the collection of all finite subset of $[0,1]$. Let $X_1, \dots, X_n\sim \nu$ be independent and identically distributed random variables with realizations $x_1,\dots,x_n$. If $A = \{x_1,\dots, x_n\}\in\mathcal{C}$, then $\nu_n(C)=1$ but $\nu(C)=0$. Thus $\sup_{A\in\mathcal{C}}\bigl(\nu_n(A)-\nu(A)\bigr) = 1$ for all $n$ and  $\mathcal{C}$ is not a Glivenko--Cantelli class.
\end{example}
So the existence of (nontrivial) Glivenko--Cantelli classes is not obvious. The Glivenko--Cantelli theorem shows that they indeed exists, with the class of rays the original \cite{glivenko1933, cantelli1933} and best known example.
\begin{definition}[Rays]
The class of \emph{open rays} and the class of \emph{closed rays} are
\[
\mathcal{R}_{\mathrm{o}}=\{(-\infty,a): a\in\mathbb{R}\}\cup\{\varnothing, \mathbb{R}\}, \quad
\mathcal{R}_{\mathrm{c}}=\{(-\infty,a]: a\in\mathbb{R}\}\cup\{\varnothing, \mathbb{R}\}
\]
respectively. Their union is the class of \emph{rays}, denoted
\[
\mathcal{R} = \mathcal{R}_{\mathrm{c}}\cup \mathcal{R}_{\mathrm{o}}.
\]
\end{definition}
These are all $\pi$-systems \cite[p.~202]{Dynkin1965} but not $\sigma$-algebras. In this article, we restrict ourselves to ``left'' rays, i.e., where the left-end point of the interval is always $-\infty$. In topology, $\mathcal{R}_{\mathrm{o}}$ is called the left-ray topology or left-order topology.

Glivenko \cite{glivenko1933} showed that $\mathcal{R}_{\mathrm{c}}$ is a Glivenko--Cantelli class  whereas Cantelli \cite{cantelli1933} showed that $\mathcal{R}_{\mathrm{o}}$ is a Glivenko--Cantelli class but we will soon see that it makes no difference whether we use $\mathcal{R}_{\mathrm{c}}$, $\mathcal{R}_{\mathrm{o}}$, or $\mathcal{R}$; all three statements are equivalent. Obviously, the name ``Glivenko--Cantelli class'' came about much later \cite{Shorack2009}, earlier statements of the result were more of the following form:
\begin{theorem}[Glivenko--Cantelli]\label{thm:GC}
Let $\nu$ be a Borel probability measure and $\nu_n$ be the corresponding empirical measure, $n \in \mathbb{N}$. Then, almost surely,
\begin{equation}\label{eq:1}
\sup_{A\in\mathcal{R}}\, \bigl( \nu_n(A)-\nu(A) \bigr) \to 0.
\end{equation}
\end{theorem}
There is no substantive difference if we instead use $\mathcal{R}_{\mathrm{o}}$ or $\mathcal{R}_{\mathrm{c}}$ in place of $\mathcal{R}$ in Theorem~\ref{thm:GC} as 
\begin{equation}\label{eq:KS}
\sup_{A\in \mathcal{R}_{\mathrm{c}}} \bigl( \mu(A)-\nu(A) \bigr) = \sup_{A\in \mathcal{R}_{\mathrm{o}}} \bigl( \mu(A)-\nu(A) \bigr) = \sup_{A\in\mathcal{R}} \,\bigl( \mu(A)-\nu(A) \bigr)
\end{equation}
for any probability measures $\mu$ and $\nu$ on $(\mathbb{R},\mathcal{B})$, a fact that follows easily  from the continuity of probability measures. The common value in \eqref{eq:KS} defines a distance between probability measures called the Kolmogorov--Smirnov distance, although it is usually defined in terms of $\mathcal{R}_{\mathrm{c}}$  \cite{kelbert2023}.

The \emph{total variation distance} is defined by	
\begin{equation}\label{eq:tv}
\D_{\tv}(\mu \Vert \nu) =\sup_{A\in \mathcal{B}} \, \bigl( \mu(A)-\nu(A)\bigr).	
\end{equation}
So the Kolmogorov--Smirnov distance is just the total variation distance ``restricted'' to $\mathcal{R} \subseteq \mathcal{B}$, i.e., a partial variation over a smaller class of sets than the full Borel $\sigma$-algebra. As we have alluded to earlier, the total variation is an example of $f$-divergence, an extensive class of distance between probability measures. Note that whether we take absolute values or not in \eqref{eq:tv} makes no difference to the value because $A\in \mathcal{B}$ iff $A^\mathsf{c} \in \mathcal{B}$.  Throughout this article, an integral sign $\int$ unadorned with upper and lower limits is taken to mean $\int_{\mathbb{R}}$.
\begin{definition}[$f$-divergence]\label{def:f}
Let $f : \mathbb{R} \to \mathbb{R}$ be a continuous convex function with $f(1)=0$. Let $\mu$ and $\nu$ be probability measures on $(\mathbb{R},\mathcal{B})$ be such that $\mu\ll \nu$, i.e., $\mu$ is absolutely continuous with respect to $\nu$, so that we may speak of Radon--Nikodym derivative $\d\mu / \d\nu$. Then the $f$-divergence \cite{renyi1961, Csiszar1964, Morimoto1963, Ali1966} between $\mu$ and $\nu$ is given by
\begin{equation}\label{eq:f}	
\D_f(\mu \Vert \nu) = \int f\Bigl(\frac{\d\mu}{\d\nu}\Bigr)\,\d\nu.
\end{equation}
\end{definition}
Different choices of $f$ yield various well-known divergences, including
Kullback--Leibler \cite{Bishop2006}, 
Le~Cam \cite{LeCam1986},
Jensen--Shannon \cite{Lin1991}, 
Jeffreys \cite{Jeffreys1998}, 
Chernoff \cite{Chernoff1952}, 
Pearson $\chi^2$ \cite{Pearson1992}, 
Hellinger squared \cite{Hellinger1909},
exponential \cite{Stanislav2017}, and 
alpha--beta \cite{Eguchi1985}
divergences, and so on. For $f(t)=\lvert t-1 \rvert /2$, we get \cite{Tsybakov2009}:
\[
\int  \frac{1}{2}\Bigl\lvert\frac{\d\mu}{\d\nu}-1\Bigr\rvert\,\d\nu  = \sup_{A\in \mathcal{B}} \, \bigl( \mu(A)-\nu(A) \bigr),
\]
the total variation distance. To obtain a Glivenko--Cantelli theorem for an arbitrary $f$-divergence, we first need an analog of Kolmogorov--Smirnov distance corresponding to an $f$-divergence.
\begin{question}\label{ques:1}
Is there a notion of $f$-divergence over $\mathcal{R}$ for general $f$ that (a) reduces to the Kolmogorov--Smirnov distance in \eqref{eq:KS} when $f(t)=\lvert t-1 \rvert /2$; and (b) reduces to the standard $f$-divergence in \eqref{eq:f} when $\mathcal{R}$ is replaced by $\mathcal{B}$?
\end{question}
The answer to Question~\ref{ques:1} is yes, provided by what we will call an $f$-divergence over $\mathcal{R}$ and denoted $\D_f^\mathcal{R}(\mu \Vert \nu)$. We will establish the existence of $f$-divergence over $\mathcal{R}$ in Section~\ref{sec:Rfdiv} and develop some of its properties in Section~\ref{sec:ineq}, showing in particular that this new notion preserves some of the best-known relations between different $f$-divergences.  In Section~\ref{sec:GC}, we will prove the result claimed in the title of our article, namely, $\D_f^\mathcal{R}(\nu_n \Vert \nu)$ converges to zero almost surely. Note that when $f(t)=\lvert t-1 \rvert /2$, this reduces to the Glivenko--Cantelli theorem in Theorem~\ref{thm:GC}.

As we will see, it takes some effort to establish the notion of $f$-divergence over $\mathcal{R}$.  For the total variation distance, getting from  \eqref{eq:tv} to \eqref{eq:KS} is a matter of simply replacing $\mathcal{B}$ by $\mathcal{R}$. But for general $f$-divergence, $\mathcal{B}$ does not even appear directly in \eqref{eq:f} and it is no longer a matter of simply  replacing $\mathcal{B}$ by  $\mathcal{R}$. Indeed we know of no straightforward answer to Question~\ref{ques:1}.

In particular, we emphasize that the groundbreaking work of Vapnik and Chervonenkis \cite{VapnikChervonenkis1971} sheds no light on whether the Glivenko--Cantelli theorem remains true when total variation distance is replaced by an $f$-divergence. They showed that if the \textsc{vc}-dimension of a class $\mathcal{C} \subseteq \Sigma$, i.e., the maximum number of points that $\mathcal{C}$ can shatter, is finite, then $\mathcal{C}$ is a Glivenko--Cantelli class, i.e., $\sup_{A\in\mathcal{C}}\bigl(\nu_n(A)-\nu(A)\bigr)\to 0$ almost surely. \label{pg:VC} More generally, existing Vapnik--Chervonenkis theory is still cast entirely in terms of the total variation distance  \cite{Vapnik1998}.  Indeed, our work in this article suggests that a ``Vapnik--Chervonenkis theory for $f$-divergence'' is likely within reach. We will establish some initial results in Section~\ref{sec:VC}.

Lastly, readers familiar with Choquet integrals \cite{choquet1953, Denneberg1994} may think that it provides a plausible answer to Question~\ref{ques:1}, as we did initially; we will explain why it does not in Section~\ref{sec:cho}.

\section{Notations and terminologies}

\subsection{Functions} We reserve the letter $f$ for $f$-divergence in this article. Other functions will be denoted $g$, $h$, $\varphi$, $\psi$. Depending on context, the notation  $X_k\uparrow X$ could mean an increasing sequence $X_1 \le X_2 \le \dots \le X_k \in \mathbb{R}$ converging to a real value $X \in \mathbb{R}$ or a nested sequence of sets  $X_1 \subseteq X_2 \subseteq \dots \subseteq X_k \subseteq \mathbb{R}$ converging to a set $X \subseteq \mathbb{R}$. Similarly for $X_k\downarrow X$.

We write $\lVert \, \cdot \, \rVert_p$ for the $L^p$-norm for all $p \in [1,\infty]$ except $p = 2$. The $L^2$-norm will just be denoted as $\lVert \, \cdot \,  \rVert$. We write $L^2(\nu) \coloneqq L^2(\mathbb{R},\mathcal{B},\nu)$ for the space of measurable functions with finite $L^2$-norm and
\[
L^2_\p (\nu)\coloneqq \{g\in L^2(\nu): g\ge 0\}
\]
for the cone of functions nonnegative almost everywhere. As usual, by ``function'' we mean an equivalence class of functions that differ at most on a measure-zero set and conditions like ``$h \ge g$'' will always be in the almost everywhere sense without specification. On the other hand, we will specify the measure every time since we often have to deal with multiple measures. 

Composition of functions will always be denoted by $\circ$ and pointwise product of real-valued functions will be denoted by $\cdot$ if necessary for emphasis or otherwise left unmarked.

Let $\varnothing \ne M\subseteq L^2(\nu)$ be a closed convex subset. The \emph{metric projection} onto $M$ is the operator $\proj_M: L^2(\nu) \to L^2(\nu)$ that takes $g \in  L^2(\nu)$ to the closest point $g_* \in M$, i.e., $\lVert g - g_* \rVert = \min_{h \in M} \lVert h - g \rVert$.
The  existence and uniqueness of $g_*$ is guaranteed by the conditions on $M$ and so $\proj_M$ is well-defined. It is also well-known that $\proj_M$ is continuous \cite[p.~52]{Alimov2021}.

\subsection{Sequences} We will remind the readers of the three notions of convergence used in our article: For a sequence of measurable functions $g_n$ converging to $g$ in $L^2(\nu)$. The convergence is said to be (i) $\nu$-almost surely if $g_n$ converges to $g$ pointwise except on a $\nu$-null set; (ii) in $\nu$-measure if $\lim_{n\to\infty} \nu(\lvert g_n-g \rvert>\varepsilon)=0$ for all $\varepsilon>0$; (iii) in $L^2$-norm if $\lVert g_n-g \rVert \to 0$. The following lemma, adapted from \cite[pp.~80--83]{Cohn2013}, summarizes the relations between them.

\begin{lemma}\label{lem:modes-of-cvg} 
Let $\nu$ be a Borel probability measure.
\begin{enumerate}[\upshape (a)]
\item If $g_n\to g$ in $L^2$-norm, then $g_n\to g$ in $\nu$-measure.
\item If $g_n\to g$ in $\nu$-measure, then there is a subsequence $(g_{n_k})$ such that $g_{n_k}\to g$ almost surely.
\item If $g_n\to g$ almost surely and there exists $g\in L^2(\nu)$ that dominates $g_n$, then $g_n\to g$ in $L^2$-norm.
\end{enumerate}
\end{lemma}

We will also need an observation that follows from $\lvert \int g_n-g\,\d\nu\rvert \le \int \lvert g_n-g \rvert\,\d\nu\le \bigl( \int (g_n-g)^p\,\d\nu\bigr)^{1/p}$.
\begin{lemma}\label{lem:lp-imply-exch} 
Let $g_n\in L^p(\nu)$. If $g_n\to g$ in $L^p$-norm for some $1\le p<\infty$, then 
\[
\int g\,\d\nu =\lim_{n\to\infty} \int g_n \,\d\nu.
\]
\end{lemma}

\subsection{Measures}\label{sec:meas}

We write $2^\Omega$ for the power set of $\Omega$. A collection of subsets $\mathcal{E} \subseteq 2^\Omega$ is called a \emph{$\pi$-system} if it is closed under finite intersections \cite[p.~202]{Dynkin1965}. The following results from \cite[p.~38]{Cohn2013} and \cite[p.~58]{Schilling2017} are reproduced here for easy reference.
\begin{lemma}\label{lem:pi-sys} 
Let $\mu,\nu$ be finite measures on $(\Omega,\Sigma)$ such that $\mu(\Omega)=\nu(\Omega)$. If $\mathcal{A}$ is a $\pi$-system that generates $\Sigma$ and $\mu(A)=\nu(A)$ for any $A\in \mathcal{A}$, then $\mu=\nu$.
\end{lemma}

\begin{lemma}\label{lem:sigma-preimage} 
Let $g:X\to Y$ be any function and $\mathcal{E}\subseteq 2^Y$. Then $\sigma(g^{-1}(\mathcal{E}))=g^{-1}(\sigma(\mathcal{E}))$.
\end{lemma}
Whenever we write $\d\mu / \d\nu$, we implicitly assume that $\mu \ll \nu$.

\section{$f$-divergence over rays}\label{sec:Rfdiv}

The goal of this section is to resolve Question~\ref{ques:1} in the affirmative. The answer is an analogue of $f$-divergence with respect to the class of rays $\mathcal{R}$ that we will call the $f$-divergence over $\mathcal{R}$. This appears only at the end of this section in Definition~\ref{def:Rfdiv}. It will take the groundwork developed over the course of the whole section before we can show that the notion is indeed well-defined. There are two milestones to watch for: We will see in Theorem~\ref{thm:main-theorem} that the construction of such a divergence would work for any class $\mathcal{C} \subseteq \mathcal{B}$ that satisfies a certain ``Radon--Nikodym property.'' We will see in Corollary~\ref{cor:RNP}, which follows from Theorem~\ref{thm:key-lemma}, that $\mathcal{C} = \mathcal{R}$ has this property.

We begin by establishing some basic properties related to $\mathcal{R}$. Throughout this section, $\nu$ denotes a Borel probability measure on $(\mathbb{R},\mathcal{B})$.

\begin{proposition}\label{prop:FR}
The set
\begin{equation}\label{eq:GR}
 G(\mathcal{R}) \coloneqq \bigl\{g\in L^2(\nu) : \{ g>r \} \in \mathcal{R} \text{ for all } r\in\mathbb{R}\bigr\}
\end{equation}
is a closed convex cone in $L^2(\nu)$ comprising all nonincreasing functions in $L^2(\nu)$.
\end{proposition}

\begin{proof}
We first show that $ G(\mathcal{R})$ is exactly the set of nonincreasing $L^2$-functions.
Let $g$ be such a function and let $r\in\mathbb{R}$. Set $a=\sup\{x:g(x)>r\}$. Any $x\in \{g>r\}$ has $x\le a$ so $\{g>r\}\subseteq (-\infty,a]$. If there is an $x<a$ with $g(x)\le r$, then there must be some  $y>x$ with $g(y)>r\ge g(x)$. This contradicts the assumption that $g$ is nonincreasing. Hence $(\infty,a)\subseteq \{g>r\}\subseteq (-\infty,a]$. It follows that $\{g>r\}\in\mathcal{R}$ and so $g\in G(\mathcal{R})$. Conversely, let $g\in G(\mathcal{R})$. Suppose we have $x<y$ with $g(x)<g(y)$. Then there is some $a$ with $(-\infty,a)\subseteq \{g>g(x)\}\subseteq (-\infty,a]$. Since $g(y)>g(x)$, it follows that $y\in\{g>g(x)\}$, and so $x<y<a$. Consequently, $x\in\{g>g(x)\}$, i.e., $g(x)>g(x)$, a contradiction. Hence $g$ must be nonincreasing.

$G(\mathcal{R})$ is a closed convex cone because nonnegative linear combinations or $L^2$-limits of nonincreasing functions remain nonincreasing. The former is routine. The latter follows from Lemma~\ref{lem:modes-of-cvg}: A sequence $(g_n)$ of nonincreasing functions with  $\lVert g_n-g \rVert \to 0$ must also converge in $\nu$-measure. So there is a subsequence $(g_{n_k})$ that converges to $g$ almost surely. So there is a null set $A$ where $\lim_{k\to\infty}g_{n_k}(x)=g(x)$ for all $x \notin A$. Since each $g_{n_k}$ is nonincreasing, $g$ is almost surely nonincreasing.
\end{proof}
 
An immediate consequence of Proposition~\ref{prop:FR} is that $ G(\mathcal{R})$ has a well-defined metric projection. For a given $g \in L^2(\nu)$, finding $\proj_{ G(\mathcal{R})} g$ is equivalent to finding the nearest nonincreasing function to $g$ in $L^2$-norm.

%The reason for not choosing $\mathcal{R}_{\mathrm{c}}$ is that $ G(\mathcal{R}_\mathrm{c})$ is not a closed set, making it unsuitable for a proper discussion of metric projections. First, any function $f\in G(\mathcal{R}_\mathrm{c})$ (except for constant function) cannot be continuous, because the preimage of the open set 
%\[
%f^{-1}(r,\infty)=\{f>r\}=(-\infty, a]
%\]
%is not open. As a concrete example, consider the sequence of functions 
%\[
%f_n=-\sum_{k\in\mathbb{Z}} \frac{k}{n}  \1_{\bigl(\frac{k}{n}, \frac{k+1}{n}\bigr]}.
%\]
%Since $f_n\in  G(\mathcal{R}_\mathrm{c})$, we have $f_n\to f$ in $L^2$-norm, where $f(x)=-x$. However, $f\notin  G(\mathcal{R}_\mathrm{c})$, demonstrating that $ G(\mathcal{R}_\mathrm{c})$ is not closed.

\begin{definition} 
A collection $\{A_1, \dots, A_n\}$ is called an \emph{ordered partition} of $\mathbb{R}$ if it satisfies the following conditions:
\begin{enumerate}[\upshape (i)]
\item disjoint: $A_i\cap A_j=\varnothing$ for any $i\ne j$;
\item cover: $A_1 \cup \dots \cup A_n =\mathbb{R}$;
\item ordered: $\sup A_i\le \inf A_j$ for each $i\le j$.
\end{enumerate}
An \emph{ordered simple function} $h:\mathbb{R}\to \mathbb{R}$ is a simple function defined on an ordered partition $\{A_1, \dots, A_n\}$.	
\end{definition}
 
It is well-known that any $g\in L^2_\p (\nu)$ can be approximated to arbitrary accuracy by an increasing sequence of nonnegative simple functions $(h_n)$. We will show  in Corollary~\ref{cor:2.7} that $h_n$ can be chosen to be ordered simple functions.
\begin{proposition}\label{prop:gf}
If $g\in L^2_\p (\nu)$ is almost surely continuous, then there exists a monotone increasing sequence of nonnegative ordered simple functions $(h_n)$ that converges pointwise to $g$ almost surely.
\end{proposition}
\begin{proof} 
Define the ordered partition
\[
\mathcal{P}_n \coloneqq \Bigl\{\Bigl[\frac{nk}{2^n}, \frac{n(k+1)}{2^n}\Bigr): k=-2^n, -2^n+1, \dots, 2^n-1\Bigr\}\cup \bigl\{(-\infty,-n), [n,\infty)\bigr\}
\]
and the ordered simple function $h_n=\sum_{A\in \mathcal{P}_n} \inf_{x\in A} g(x)\1_A$. Since $g$ is almost surely continuous, we can find a null set $S$ such that for any $\varepsilon > 0$ and $x \in \mathbb{R} \backslash S$, there exist $n \in \mathbb{N}$ and $A \in \mathcal{P}_n$ with $x \in A \subseteq B_\varepsilon(x)$. Consequently,
\[
\inf_{y\in B_\varepsilon(x)}g(y)\le \inf_{y\in A}g(y)=h_n(x)\le g(x),
\]
where $\inf_{y\in B_\varepsilon(x)}g(y)\uparrow g(x)$ as $\varepsilon\downarrow 0$ since $g$ is continuous on $\mathbb{R}\backslash S$. Therefore, $h_n(x)\to g(x)$ almost surely as $n\to\infty$.
For any $n \le m$, we have $A\subseteq B$ for any $A\in\mathcal{P}_m$ and $B\in\mathcal{P}_n$; so $h_n(x)\le h_m(x)$ for all $x\in\mathbb{R}$.
\end{proof}
In Proposition~\ref{prop:gf}, $h_n$ is dominated by $g$; so we also have $h_n\to g$ in $L^2$-norm by Lemma~\ref{lem:modes-of-cvg}.
The almost sure continuity in Proposition~\ref{prop:gf} may sometimes be overly restrictive. We derive an alternative version without this requirement. 
\begin{corollary}\label{cor:2.7} 
Let $g\in L^2_\p (\nu)$. Then there exists a sequence of bounded ordered simple functions $(h_n)$ that converges to $g$ in $L^2$-norm and almost surely.
\end{corollary}
\begin{proof}
Given that the space of continuous functions with compact support $C_{\mathrm{c}}(\mathbb{R})$ is dense in $L^2(\nu)$ \cite[p.~354]{Hunter2001}, there is a continuous function $\varphi$ with compact support that approximates $g$ to accuracy $\lVert g-\varphi \rVert <\varepsilon/2$ for any $\varepsilon>0$. On the other hand, the continuous function $\varphi$ is in $L^2(\nu)$; so by Proposition~\ref{prop:gf} there is an ordered simple function $h$ that approximates it to accuracy $\lVert \varphi- h\rVert <\varepsilon/2$. Hence $\lVert h - g\rVert  <\varepsilon$. This shows that for any $g\in L^2_\p (\nu)$, there exists a sequence of ordered simple function $h_n$ such that $h_n\to g$ in $L^2$-norm. By passing through a subsequence if necessary, Lemma~\ref{lem:modes-of-cvg} ensures $h_n\to g$ almost surely as well.
\end{proof} 
We next see that the metric projection of an ordered simple function onto $G(\mathcal{R})$  remains ordered simple.
\begin{proposition}[Projection]\label{prop:proj}
Let $h=\sum_{i=1}^{n} \alpha_i\1_{A_i}\in L^2_\p (\nu)$ be an ordered simple function with $\{A_1,\dots,A_n\}$ an ordered partition of $\mathbb{R}$. Then $\proj_{ G(\mathcal{R})} h$ is also an ordered simple function on $\{A_1, \dots, A_n\}$.
\end{proposition}
\begin{proof} 
Suppose that $\proj_{ G(\mathcal{R})} h$ is not constant on $A_j$ for some $j=1,\dots,n$. Define 
\[
\varphi = 
\begin{cases}
\proj_{ G(\mathcal{R})} h & \text{on } A_j^{\mathsf{c}},  \\
c & \text{on } A_j,	
\end{cases}
\]
where $c=\proj_{ G(\mathcal{R})} h (\omega_*)$ and $\omega_*=\argmin_{\omega\in A_j}\lvert h(\omega) -\proj_{ G(\mathcal{R})} h(\omega)\rvert$. Then $\varphi\in G(\mathcal{R})$ and
\[
\int  (h-\proj_{ G(\mathcal{R})} h)^2\,\d\nu \ge \int  (h-\varphi)^2\,\d\nu,
\]
a contradiction.
\end{proof}
A consequence of Proposition~\ref{prop:proj} is that metric projection onto $G(\mathcal{R})$ preserves ordering.
\begin{proposition}[Monotonicity]\label{prop:proj-monotone} 
Let  $g_1,g_2\in L^2_\p (\nu)$. If $g_1\le g_2$, then
\[
\proj_{ G(\mathcal{R})} g_1 \le \proj_{ G(\mathcal{R})} g_2.
\]
\end{proposition}

\begin{proof} 
We first show that the statement holds for ordered simple functions. Let $h_1=\sum_{i=1}^{n} a_i\1_{A_i}$ and $h_2=\sum_{j=1}^{m} b_j\1_{B_j}$ where $(A_1, \dots, A_n)$ and $(B_1, \dots, B_m)$ are ordered partitions.  Suppose $h_1\le h_2$ but that $\proj_{ G(\mathcal{R})} h_1 > \proj_{ G(\mathcal{R})} h_2$ on $E=A_i\cap B_j$ for some $i$ and $j$. Define
\begin{align*}
\varphi &=(\proj_{ G(\mathcal{R})} h_1)\cdot\1_{E^{\mathsf{c}} }+(\proj_{ G(\mathcal{R})} h_2)\cdot\1_E ,\\
\psi&=(\proj_{ G(\mathcal{R})} h_2)\cdot\1_{E^{\mathsf{c}} }+(\proj_{ G(\mathcal{R})} h_1)\cdot\1_E .
\end{align*}
Both $\varphi$ and $\psi$ are clearly nonincreasing and thus $\mathcal{R}$-measurable. There are six possibilities for the values of these functions on $E$: 
\begin{multicols}{2}
\begin{enumerate}[\upshape (a)]
    \item \label{case:1} $ h_2\ge h_1\ge \proj_{ G(\mathcal{R})} h_1\ge \proj_{ G(\mathcal{R})} h_2 $,
    \item \label{case:2} $ h_2\ge \proj_{ G(\mathcal{R})} h_1\ge h_1\ge \proj_{ G(\mathcal{R})} h_2 $,
    \item \label{case:3} $ h_2\ge \proj_{ G(\mathcal{R})} h_1\ge \proj_{ G(\mathcal{R})} h_2\ge h_1 $,
    \item \label{case:4} $ \proj_{ G(\mathcal{R})} h_1\ge h_2\ge h_1\ge \proj_{ G(\mathcal{R})} h_2 $,
    \item \label{case:5} $ \proj_{ G(\mathcal{R})} h_1\ge h_2\ge \proj_{ G(\mathcal{R})} h_2\ge h_1 $,
    \item \label{case:6} $ \proj_{ G(\mathcal{R})} h_1\ge \proj_{ G(\mathcal{R})} h_2\ge h_2\ge h_1 $.
\end{enumerate}
\end{multicols}
\noindent For Cases~\ref{case:1}, \ref{case:2}, and \ref{case:3},
\[
\lVert \psi-h_2 \rVert \le \lVert \proj_{ G(\mathcal{R})} h_2-h_2  \rVert;
\]
for Cases~\ref{case:5} and \ref{case:6},
\[
\lVert \varphi-h_1 \rVert \le \lVert \proj_{ G(\mathcal{R})} h_1-h_1 \rVert;
\] 
and for Case~\ref{case:4}, either of the following
\[
\lVert \psi-h_2 \rVert \le \lVert \proj_{ G(\mathcal{R})} h_2-h_2 \rVert, \quad
\lVert \varphi-h_1 \rVert \le \lVert \proj_{ G(\mathcal{R})} h_1-h_1 \rVert
\]
must hold; all leading to contradictions.  

For the general case with $g_1,g_2\in L^2_\p (\nu)$, we set $g_0\coloneqq g_2-g_1\in L^2_\p (\nu)$. By Corollary~\ref{cor:2.7}, we have two sequences of nonnegative ordered simple functions $(h_{0,n})$ and $(h_{1,n})$ such that $h_{0,n}\to g_0$ and $h_{1,n}\to g_1$ in $L^2$-norm. Set $h_{2,n}\coloneqq h_{0,n}+h_{1,n}$, which is nonnegative, ordered simple, and satisfies $h_{2,n}\to g_2$ in $L^2$-norm and $h_{2,n}\ge h_{1,n}$ for each $n$.  By the continuity of the metric projection operator $\proj_{ G(\mathcal{R})}$,
\[
\proj_{ G(\mathcal{R})} g_1 = \lim_{n\to\infty} \proj_{ G(\mathcal{R})} h_{1,n} \le \lim_{n\to\infty} \proj_{ G(\mathcal{R})} h_{2,n} = \proj_{ G(\mathcal{R})} g_2,
\]
as required.
\end{proof}
A slight refinement of the last step of the proof above yields the following.
\begin{corollary}[Sequential monotonicity]
Let $g\in L^2_\p (\nu)$. If $(h_n)$ is a monotone increasing sequence in $L^2_\p (\nu)$ converging to $g$, then $(\proj_{ G(\mathcal{R})} h_n)$  is a monotone increasing sequence in $L^2_\p (\nu)$ converging to $\proj_{ G(\mathcal{R})} g$. Convergence is both in the sense of $L^2$-norm and almost surely.
\end{corollary}
\begin{proof}
Let $h_n$ be a sequence of ordered simple function such that $h_n\to g$ almost surely and in $L^2$-norm. Continuity of $\proj_{ G(\mathcal{R})}$ yields
\[
\proj_{ G(\mathcal{R})} h_n\to \proj_{ G(\mathcal{R})} g
\]
in $L^2$-norm. By passing to a subsequence if necessary, we may assume that this convergence also holds in the almost sure sense.
\end{proof}

\begin{proposition}\label{prop:proj-indicator-ineq} 
Let $A\in \mathcal{R}$. Then
\[
\proj_{ G(\mathcal{R})} (g\1_A) = \proj_{ G(\mathcal{R})} (g\1_A) \cdot \1_A
\le (\proj_{ G(\mathcal{R})} g)\cdot \1_A.
\]
\end{proposition}

\begin{proof} 
We have
\begin{align*}
\int (g\1_A-&\proj_{ G(\mathcal{R})} (g\1_A))^2\,\d\nu\\
&= \int_A (g\1_A-\proj_{ G(\mathcal{R})} (g\1_A))^2\,\d\nu + \int_{A^{\mathsf{c}} } (g\1_A-\proj_{ G(\mathcal{R})} (g\1_A))^2\,\d\nu\\
&= \int_A (g\1_A-\proj_{ G(\mathcal{R})} (g\1_A))^2\,\d\nu  + \int_{A^{\mathsf{c}} } (\proj_{ G(\mathcal{R})} (g\1_A))^2\,\d\nu\\
&\ge  \int_A (g\1_A-\proj_{ G(\mathcal{R})} (g\1_A))^2\,\d\nu\\
&=\int  (g\1_A-\proj_{ G(\mathcal{R})} (g\1_A)\cdot\1_A)^2\,\d\nu.
\end{align*}
Since the projection is unique, we have $\proj_{ G(\mathcal{R})} (g\1_A) = \proj_{ G(\mathcal{R})} (g\1_A) \cdot \1_A$. 
As $g\1_A\le g$, the monotonicity in Proposition~\ref{prop:proj-monotone}  yields
\[
\proj_{ G(\mathcal{R})} (g\1_A) = \proj_{ G(\mathcal{R})} (g\1_A) \cdot\1_A \le (\proj_{ G(\mathcal{R})} g)\cdot\1_A. \qedhere
\]
\end{proof}

Let  $\mathcal{C} \subseteq \mathcal{B}$ be a $\sigma$-subalgebra. The Radon--Nikodym theorem states that for $\sigma$-finite measures $\mu$ and $\nu$ on the measurable space $(\mathbb{R},\mathcal{B})$ with $\mu\ll\nu$, there exists a $\mathcal{C}$-measurable function $\rho_\mathcal{C}$ such that $\mu(A)=\int_A \rho_\mathcal{C} \,\d\nu$ for any $A\in\mathcal{C}$.

What happens if $\mathcal{C}$ is not a $\sigma$-subalgebra? For us, the most important case to consider is when $\mathcal{C}=\mathcal{R}$, which is not a $\sigma$-subalgebra. We will show in Corollary~\ref{cor:RNP} that $\mathcal{R}$ satisfies a kind of Radon--Nikodym property: There is a $\mathcal{R}$-measurable function $\rho_\mathcal{R}$ and a subcollection $\mathcal{E}\subseteq \mathcal{R}$  such that  $\mu(A)=\int_A \rho_\mathcal{R}\,\d\nu$ for any $A\in \mathcal{E}$. The bulk of the work is in establishing the following technical result.

\begin{theorem}\label{thm:key-lemma} 
Let $g\in L^2_\p (\nu)$. Define the $\pi$-system $\mathcal{E}(g) \subseteq \mathcal{R}$ by
\begin{equation}\label{eq:E(g)}
\mathcal{E}(g) \coloneqq \bigl\{ \{\proj_{ G(\mathcal{R})} g>r \} : r\ge 0 \bigr\}\cup \{\mathbb{R} \}.
\end{equation}
Then  for any $E\in \mathcal{E}(g)$,
\begin{equation}\label{eq:3}
\int_E  \proj_{ G(\mathcal{R})} g\,\d\nu = \int_E  g\,\d\nu.
\end{equation}
\end{theorem}

\begin{proof} 
We establish the result in four steps, showing that
\begin{enumerate}[\upshape (i)]
\item\label{it:sim} it holds when $g$ is an ordered simple function $h$;
\item\label{it:ass} it holds for $E=\{\proj_{ G(\mathcal{R})} g>r\}$ under the assumption that $\nu(\{\proj_{ G(\mathcal{R})} g= r\}) = 0$;
\item\label{it:ext} it extends to all $E=\{\proj_{ G(\mathcal{R})} g>r\}$;
\item\label{it:R} it holds for $E=\mathbb{R}$.
\end{enumerate}

\textsc{Step~\ref{it:sim}:} First we assume that we have an ordered simple function
\[
h=\sum_{i=1}^{n} \alpha_i\1_{A_i},\quad \proj_{ G(\mathcal{R})} h=\sum_{i=1}^{n} \beta_i\1_{A_i}.
\]
So $E=\{\proj_{ G(\mathcal{R})} h>r\}$ can be written as $E=A_1 \cup \dots \cup A_k \in\mathcal{E}(h)$ where $\beta_k>r$ and $\beta_{k+1}\le r$. Write $\alpha=(\alpha_1, \dots, \alpha_n)$, $\beta =(\beta_1, \dots, \beta_n)$, $\omega=(\omega_1, \dots, \omega_n)$ where $\omega_i=\nu(A_i)$, and let $W=\diag(\omega)$. Then $\beta$ is the solution to the following quadratic programming problem:
\begin{align*}
	\operatorname{minimize} &\quad (\alpha -\beta)^\tp W (\alpha-\beta) \\
	\operatorname{subject~to} &\quad \beta_1\ge \dots\ge \beta_n\ge 0.
\end{align*}
Observe that for
\[
\Pi = 
\begin{bmatrix}
  1    &     -1    &        &        \\ 
      &   1      &   -1    &       \\ 
      &         &  \ddots     & -1      \\
      &     &  &      1           
\end{bmatrix},\qquad
\Pi^{-1} = 
\begin{bmatrix}
  1    &     1    &  \dots      &   1     \\ 
      &   1      &   \dots    &   1    \\ 
      &         &  \ddots     & \vdots      \\
      &     &     &      1           
\end{bmatrix}.
\]
Let $\gamma=\Pi\beta$. Then the quadratic programming problem above is equivalent to
\begin{align*}
	\operatorname{minimize} &\quad \gamma^\tp \Pi^{-\tp}W\Pi^{-1} \gamma - 2 \alpha^\tp \Pi^{-1} W \gamma \\
	\operatorname{subject~to} &\quad -\gamma\preceq  0.
\end{align*}
Standard KKT conditions yield
 \[
 \lambda = 2 \Pi^{-\tp}W(\Pi^{-1}\gamma -\alpha)=2 \Pi^{-\tp}W(\beta-\alpha)
 \] 
or
\begin{equation}\label{eq:4}
\begin{bmatrix} 
     \lambda_1      \\ 
      \lambda_2     \\
      \vdots \\
      \lambda_n
\end{bmatrix}
=2 
\begin{bmatrix} 
  1    &         &        &        \\ 
  1    &   1      &       &       \\ 
  \vdots    &  \vdots       &  \ddots     &       \\
   1   &  1   & \dots &      1           
\end{bmatrix}
\begin{bmatrix} 
  \omega_1    &         &        &        \\ 
      &   \omega_2      &       &       \\ 
      &         &  \ddots     &       \\
      &     &  &      \omega_n           
\end{bmatrix}
\begin{bmatrix} 
     \beta_1-\alpha_1      \\ 
     \beta_2-\alpha_2    \\
      \vdots \\
      \beta_n-\alpha_n
\end{bmatrix}.	
\end{equation}
Since $\beta_k = (\Pi^{-1} \gamma)_k = \gamma_k+\gamma_{k+1}+\dots+\gamma_n>r$ and $\beta_{k+1}=\gamma_{k+1}+\gamma_{k+2}+\dots+\gamma_n\le r$, we must have $\gamma_k>0$. By the KKT complementary slackness condition,
\begin{equation}\label{eq:lk}
\lambda_k=2( \Pi^{-\tp}W(\beta-\alpha))_k=2\sum_{i=1}^k \omega_i(\beta_i-\alpha_i) =0.
\end{equation}
Hence we have
\[
\int  (\proj_{ G(\mathcal{R})} h)\cdot\1_E \,\d\nu = \sum_{i=1}^k\beta_i \omega_i = \sum_{i=1}^k\alpha_i \omega_i = \int  h \1_E \,\d\nu.
\]

\textsc{Step~\ref{it:ass}:} Next we drop the ordered simple assumption and just require that  $g\in L^2_\p (\nu)$ be a general function with
\begin{equation}\label{eq:FR0}
\nu(\{\proj_{ G(\mathcal{R})} g = r\}) = 0.
\end{equation}
Let $(h_n)$ be a sequence of nonnegative ordered simple functions with $h_n\to g$ almost surely and in $L^2$-norm. So $\proj_{ G(\mathcal{R})} h_n\to \proj_{ G(\mathcal{R})} g$ in $L^2$-norm and thus in measure. By passing through a subsequence if necessary, we may assume that $\proj_{ G(\mathcal{R})} h_n\to \proj_{ G(\mathcal{R})} g$ almost surely and in $L^2$-norm. Set $E=\{\proj_{ G(\mathcal{R})} g>r\}$ and $E_n=\{\proj_{ G(\mathcal{R})} h_n>r\}$, bearing in mind that we assume \eqref{eq:FR0}. Then
\begin{align*}
	\lVert \1_E -\1_{E_n} \rVert ^2 &= \int (\1_E -\1_{E_n})^2\,\d\nu \\
	&=\nu(E)+\nu(E_n)-2\nu(E\cap E_n) = \nu(E\triangle E_n) = 
	\nu(E\cap E_n^{\mathsf{c}} ) + \nu(E_n\cap E^{\mathsf{c}} ).
\end{align*}
Consider the null set $N=\{\proj_{ G(\mathcal{R})} h_n\not\to \proj_{ G(\mathcal{R})} g\}$. If
\[
x\in E\cap E_n^{\mathsf{c}} =\{\proj_{ G(\mathcal{R})} g>r\}\cap\{\proj_{ G(\mathcal{R})} h_n\le r\}
\]
infinitely often, then 
\[
\liminf_{n\to\infty} \proj_{ G(\mathcal{R})} h_n(x)\le r < \proj_{ G(\mathcal{R})} g(x),
\]
and so $x\in N$. Consequently, $\limsup_{n\to\infty} E\cap E_n^{\mathsf{c}}  \subseteq N$.  On the other hand, we may decompose
\begin{align*}
E_n\cap E^{\mathsf{c}} &=\bigl[ {\{\proj_{ G(\mathcal{R})} h_n>r\}\cap \{\proj_{ G(\mathcal{R})} g<r\}}\bigr] \cup \bigl[\{\proj_{ G(\mathcal{R})} h_n>r\}\cap \{\proj_{ G(\mathcal{R})} g = r\}\bigr] \\
&\eqqcolon F_n \cup G_n.
\end{align*}
Since $G_n$ is a null set for each $n$, if $x\in F_n$ infinitely often, then 
\[
\limsup_{n\to\infty} \proj_{ G(\mathcal{R})} h_n(x)\ge r > \proj_{ G(\mathcal{R})} g(x),
\]
and so $x\in N$ as well. Consequently, $\limsup_{n\to\infty} F_n \subseteq N$. We obtain
\begin{align*}
	\limsup_{n\to\infty} \lVert \1_E -\1_{E_n} \rVert ^2 & \le \limsup_{n\to\infty} \bigl[ \nu(E\cap E_n^{\mathsf{c}} ) + \nu(F_n) + \nu(G_n)\bigr] \\
	&\le  \limsup_{n\to\infty} \nu(E\cap E_n^{\mathsf{c}} ) + \limsup_{n\to\infty} \nu(F_n)\\
	&\le \nu \Bigl(\limsup_{n\to\infty} E\cap E_n^{\mathsf{c}} \Bigr) + \nu\Bigl(\limsup_{n\to\infty} F_n\Bigr) = 0,
\end{align*}
and by H\"{o}lder,
\begin{align*}
\lVert h\1_E-h_n\1_{E_n}\rVert_1 &= \lVert h\1_E - h\1_{E_n} + h\1_{E_n}-h_n\1_{E_n}\rVert_1 \\
&\le \lVert h\cdot(\1_E - \1_{E_n})\rVert_1  + \lVert (h-h_n)\cdot\1_{E_n}\rVert_1  \\
&\le \lVert h \rVert \lVert \1_E - \1_{E_n}\rVert + \lVert h-h_n\rVert \lVert \1_{E_n}\rVert \to 0.
\end{align*}
The same argument also yields
\[
(\proj_{ G(\mathcal{R})} h_n)\cdot\1_{E_n}\to (\proj_{ G(\mathcal{R})} g)\cdot \1_E
\]
in $L^1$-norm. By Lemma~\ref{lem:lp-imply-exch}, we get
\[
	\int  g \1_E \,\d\nu = \lim_{n\to\infty} \int  h_n \1_{E_n} \,\d\nu
	=\lim_{n\to\infty} \int  (\proj_{ G(\mathcal{R})} h_n)\cdot \1_{E_n} \,\d\nu
	=\int  (\proj_{ G(\mathcal{R})} g)\cdot\1_E\,\d\nu
\]
as required.

\textsc{Step~\ref{it:ext}:} We next omit the condition \eqref{eq:FR0}. Let $x=\inf \{\proj_{ G(\mathcal{R})} g = r\}$ and assume $x>-\infty$ for nontriviality. For any $y<x$, we have $\proj_{ G(\mathcal{R})} g(y)\ge r$ and so
\begin{equation}\label{eq:gyr}
\inf_{y\in (-\infty,x)}\proj_{ G(\mathcal{R})} g(y)\ge r.
\end{equation}
If we have strict inequality in \eqref{eq:gyr}, then there exists $\varepsilon>0$ such that $\inf_{y\in (-\infty,x)}\proj_{ G(\mathcal{R})} g(y) > r+\varepsilon$. Thus 
\[
E_r\coloneqq \{\proj_{ G(\mathcal{R})} g>r\}=\{\proj_{ G(\mathcal{R})} g>r+\varepsilon\}\eqqcolon E_{r+\varepsilon}.
\]
Since $\proj_{ G(\mathcal{R})} g\ne r+\varepsilon$ almost surely, $\nu(\{\proj_{ G(\mathcal{R} )} g=r+\varepsilon\})=0$, and
\[
\int_{E_r} \proj_{ G(\mathcal{R})} g\,\d\nu = \int_{E_{r+\varepsilon}} \proj_{ G(\mathcal{R})} g\,\d\nu 
=\int_{E_{r+\varepsilon}} g\,\d\nu = \int_{E_r}  g \,\d\nu.
\]
If we have equality in \eqref{eq:gyr}, then for any $k\in\mathbb{N}$, we may pick $x_k\in (-\infty,x)$ such that 
\[
r<\proj_{ G(\mathcal{R})} g(x_k)<r+\frac{1}{k}
\]
and $r_k\in (r,\proj_{ G(\mathcal{R})} g(x_k)]$ such that
\[
r<r_k<r+\frac1k \quad\text{and}\quad \nu(\proj_{ G(\mathcal{R})} g=r_k)=0.
\]
Hence $E_k \coloneqq \{\proj_{ G(\mathcal{R})} g>r_k\}$, $k \in \mathbb{N}$, is an increasing sequence of nested sets with $E_k\subseteq E_r$ and so $\bigcup_{k=1}^\infty E_k\subseteq E_r$.  Conversely, if $y\in E_r$, i.e., $\proj_{ G(\mathcal{R})} g(y)>r$, then there exists $r_k$ such that $r_k<\proj_{ G(\mathcal{R})} g(y)$; so $y\in E_k$; and so $E_k\uparrow E_r$.  Define the finite measures
\[
\zeta(A)\coloneqq \int_A \proj_{ G(\mathcal{R})} g\,\d\nu ,\qquad \xi(A)\coloneqq \int_A g\,\d\nu
\] 
for any $A\in\mathcal{B}$. By the continuity of measure from below,
\[
\zeta(E_r) = \lim_{k\to\infty} \zeta(E_k) = \lim_{k\to\infty} \xi(E_k) = \xi(E_r)
\]
as required.

\textsc{Step~\ref{it:R}:} It remains to treat the case $E = \mathbb{R}$.
We will show that $\int \proj_{ G(\mathcal{R})} g\,\d\nu = \int  g\,\d\nu$ by first establishing it for an ordered simple function  $h=\sum_{i=1}^{n}\alpha_i\1_{A_i}$. We may assume that $\beta_k>0$ and $\beta_{k+1}=\beta_{k+2}=\dots=\beta_n=0$ without loss of generality. By \eqref{eq:4} and Step~\ref{it:sim},
\begin{align*}
\lambda_n &= 2 \sum_{i=1}^{n} \omega_i(\beta_i-\alpha_i) = 2 \sum_{i=1}^{k} \omega_i(\beta_i-\alpha_i) + 2\sum_{j=k+1}^n \omega_i (\beta_i-\alpha_i) \\
&=	-\lambda_k -2 (\omega_{k+1}\alpha_{k+1}+\omega_{k+2}\alpha_{k+2}+\dots+\omega_n\alpha_n).
\end{align*}
By \eqref{eq:lk}, $\lambda_k = 0$. By KKT dual feasibility, $\lambda_n\ge 0$. Since we also have $\omega_i,\alpha_i\ge 0$, $i=k+1, k+2,\dots,n$,
\[
\lambda_n=-2 (\omega_{k+1}\alpha_{k+1}+\omega_{k+2}\alpha_{k+2}+\dots+\omega_n\alpha_n)=0.
\] 
It follows from \eqref{eq:4} that $0=\lambda_n=2\sum_{i=1}^{n} \omega_i(\beta_i-\alpha_i)$ and so
\[
\int  \proj_{ G(\mathcal{R})} h\,\d\nu = \sum_{i=1}^{n} \beta_iw_i=\sum_{i=1}^{n} \alpha_iw_i = \int  h\,\d\nu.
\]
For a general function $g\in L^2_\p (\nu)$, pick a sequence of nonnegative ordered simple functions $h_n\to g$ in $L^2$-norm and apply Lemma~\ref{lem:modes-of-cvg} to get
\[
	\int  g \,\d\nu = \lim_{n\to\infty} \int  h_n \,\d\nu  
	=\lim_{n\to\infty} \int  \proj_{ G(\mathcal{R})} h_n \,\d\nu=\int  \proj_{ G(\mathcal{R})} g \,\d\nu. \qedhere
\]
\end{proof}
We state an immediate corollary of Theorem~\ref{thm:key-lemma}.
\begin{corollary}\label{cor:generate}
Let  $g\in L^2_\p (\nu)$. If $\proj_{ G(\mathcal{R})} g (E) \in \mathcal{B}$, then
\[
\int_E  \proj_{ G(\mathcal{R})} g\,\d\nu = \int_E  g \,\d\nu.	
\]	
\end{corollary}
\begin{proof} 
Theorem~\ref{thm:key-lemma} says that \eqref{eq:3} holds for any $E\in\mathcal{E}(g)$. The left and right sides of \eqref{eq:3} define finite measures $\zeta$ and $\xi$ respectively, and $\zeta(\mathbb{R})=\xi(\mathbb{R})$ as in Step~\ref{it:ext} of the proof of Theorem~\ref{thm:key-lemma}. Since  $\mathcal{E}(g)$ is a $\pi$-system, $\sigma(\mathcal{E}(g))=(\proj_{ G(\mathcal{R})} g)^{-1}(\mathcal{B})$ by Lemmas~\ref{lem:pi-sys} and \ref{lem:sigma-preimage}. Hence \eqref{eq:3} holds for all $E\in (\proj_{ G(\mathcal{R})} g)^{-1}(\mathcal{B})$.
\end{proof}

We now arrive at the Radon--Nikodym property for $\mathcal{R}$ mentioned earlier.
\begin{corollary}[Radon--Nikodym over rays]\label{cor:RNP}
Let $\mu,\nu$ be Borel probability measures on $(\mathbb{R},\mathcal{B})$. Then the $\mathcal{R}$-measurable function 
\[
\rho_\mathcal{R}=\proj_{ G(\mathcal{R})} \frac{\d\mu}{\d\nu},
\]
and the $\pi$-system $\mathcal{E}(\d\mu/\d\nu)=\bigl\{\{\rho_\mathcal{R} >r\}: r\ge 0 \bigr\}\cup \{\mathbb{R} \} \subseteq \mathcal{R}$ satisfy
\[
\mu(E)=\int_E \rho_\mathcal{R}\,\d\nu
\]
for any $E\in \mathcal{E}(\d\mu/\d\nu)$.
\end{corollary}
\begin{proof}
For any $E\in \mathcal{E}(\d\mu/\d\nu)$, it follows from Theorem~\ref{thm:key-lemma} that
\[
\mu(E) = \int_E \frac{\d\mu}{\d\nu}\,\d\nu
= \int_E \proj_{ G(\mathcal{R})} \frac{\d\mu}{\d\nu}\,\d\nu
= \int_E \rho_\mathcal{R}\,\d\nu.\qedhere
\]
\end{proof}

We will show that the following definition of $f$-divergence over $\mathcal{R}$ provides an affirmative answer to Question~\ref{ques:1}: Theorem~\ref{thm:main-theorem} shows that (a) when $f$ is chosen to be  $f(t) = \lvert t - 1 \rvert/2$, we recover the Kolmogorov--Smirnov distance in \eqref{eq:KS}; and Definition~\ref{def:Rfdiv} shows that by replacing $\mathcal{R}$ by $\mathcal{B}$, we recover the standard $f$-divergence in \eqref{eq:f}. A reminder of our convention in Section~\ref{sec:meas}: we assume that $\mu\ll \nu$ whenever we write $\d \mu / \d \nu$. This is not an additional imposition; it is also a requirement for standard $f$-divergence in Definition~\ref{def:f}.
\begin{definition}[$f$-divergence over rays]\label{def:Rfdiv}
Let $\mu$ and $\nu$ be Borel probability measures and $\mathcal{R}$ be the class of rays. Then the $f$-divergence over $\mathcal{R}$ is defined as
\begin{equation}\label{eq:Rfdiv}
\D_f^\mathcal{R} (\mu \Vert \nu) \coloneqq \int  f\Bigl(\proj_{ G(\mathcal{R})} \frac{\d\mu}{\d\nu}\Bigr)\,\d\nu.
\end{equation}
\end{definition}
Henceforth we will use this terminology in conjunction with standard nomenclatures. For example when we speak of total variation or Hellinger distance over $\mathcal{R}$ we mean the $f$-divergence over $\mathcal{R}$ with $f(t) = \lvert t - 1 \rvert/2$ or $f(t) = (\sqrt{t} - 1)^2$ respectively.
Obviously, if we replace $\mathcal{R}$ in \eqref{eq:Rfdiv} by the Borel $\sigma$-algebra $\mathcal{B}$, then $\D_f^\mathcal{B} (\mu \Vert \nu) = \D_f(\mu \Vert \nu)$ and we recover the standard $f$-divergence in Definition~\ref{def:f}. So we may view the standard definition as ``$f$-divergence over $\mathcal{B}$.'' 

We will provide abundant evidence in Section~\ref{sec:ineq} that our $f$-divergence over $\mathcal{R}$ in Definition~\ref{def:Rfdiv} retains almost all known properties of the standard $f$-divergence. But our most rudimentary justification is Theorem~\ref{thm:main-theorem}, i.e., the total variation distance on $\mathcal{R}$ is exactly the Kolmogorov--Smirnov distance. To get there we need two  corollaries.
\begin{corollary}\label{cor:integral-proj-ineq} 
Let $g\in L^2_\p (\nu)$. Then for any $A\in\mathcal{R}$,
\[
\int_A \proj_{ G(\mathcal{R})} g \,\d\nu\ge \int_A g \,\d\nu.
\]
\end{corollary}

\begin{proof} 
By Proposition~\ref{prop:proj-indicator-ineq} and Theorem~\ref{thm:key-lemma}, we have
\[
\int_A \proj_{ G(\mathcal{R})} g \,\d\nu = \int  (\proj_{ G(\mathcal{R})} g)\cdot \1_A \,\d\nu \ge  \int  \proj_{ G(\mathcal{R})} (g \1_A) \,\d\nu =\int_A g \,\d\nu. \qedhere
\]
\end{proof}
The $\pi$-system $\mathcal{E}(g)$ as defined in \eqref{eq:E(g)} has the following property.
\begin{corollary} 
Let $g\in L^2_\p (\nu)$. Then for any $E\in\mathcal{E}(g)$,
\[
(\proj_{ G(\mathcal{R})} g) \cdot\1_E = \proj_{ G(\mathcal{R})} (g\1_E).
\]
\end{corollary}

\begin{proof} 
By Theorem~\ref{thm:key-lemma},
\[
\int (\proj_{ G(\mathcal{R})} g)\cdot\1_E \,\d\nu = \int g\1_E \,\d\nu = \int \proj_{ G(\mathcal{R})} (g\1_E) \,\d\nu.
\]
Since  $(\proj_{ G(\mathcal{R})} g)\cdot\1_E \ge  \proj_{ G(\mathcal{R})} (g\1_E)$, we must in fact have equality.
\end{proof}

We now arrive at the main result of this section. The key to Theorem~\ref{thm:main-theorem} is Theorem~\ref{thm:key-lemma}. If we have an analogue of Theorem~\ref{thm:key-lemma} for a class of sets $\mathcal{C}$, then we have an ``$f$-divergence on $\mathcal{C}$'' for which an analogue of Theorem~\ref{thm:main-theorem} with $\mathcal{C}$ in place of $\mathcal{R}$ holds.
\begin{theorem}[Kolmogorov--Smirnov as special case]\label{thm:main-theorem}
Let $\mu,\nu$ be Borel probability measures on $(\mathbb{R},\mathcal{B})$. Then
\begin{equation}\label{eq:ks}
\int \frac{1}{2} \Bigl\lvert \proj_{ G(\mathcal{R})} \frac{\d\mu}{\d\nu}-1\Bigr\rvert \,\d\nu = \sup_{A\in \mathcal{R}} \, \bigl( \mu(A)-\nu(A) \bigr).
\end{equation}
\end{theorem}

\begin{proof} 
Let $E=\{\proj_{ G(\mathcal{R})} \d\mu/\d\nu > 1\}$. By  Theorem~\ref{thm:key-lemma} and Corollary~\ref{cor:RNP},
\[
\int_A \proj_{ G(\mathcal{R})} \frac{\d\mu}{\d\nu}\,\d\nu = \int_A \frac{\d\mu}{\d\nu} \,\d\nu = \mu(A)
\] 
for any  $A\in \mathcal{E}(\d\mu / \d\nu)$. So
\begin{align}
\int \frac{1}{2} \Bigl\lvert \proj_{ G(\mathcal{R})} \frac{\d\mu}{\d\nu}-1\Bigr\rvert \,\d\nu &=
\frac{1}{2}\int_E \proj_{ G(\mathcal{R})} \frac{\d\mu}{\d\nu}-1 \,\d\nu + \frac{1}{2}\int_{E^{\mathsf{c}} }  1- \proj_{ G(\mathcal{R})} \frac{\d\mu}{\d\nu} \,\d\nu \notag	\\
&=\frac{1}{2}\Bigl[2\int_E \proj_{ G(\mathcal{R})} \frac{\d\mu}{\d\nu}\,\d\nu - 1\Bigr] - \frac{1}{2}\Bigl[ 2\int_E \,\d\nu -1\Bigr] \notag \\
&=\int_E \proj_{ G(\mathcal{R})} \frac{\d\mu}{\d\nu}\,\d\nu - \nu(E) \label{eq:same} \\
&=\mu(E)-\nu(E) \le \sup_{A\in \mathcal{R}}\, \bigl(\mu(A)-\nu(A)\bigr), \notag
\end{align}
noting that $E=\{\proj_{ G(\mathcal{R})} \d\mu / \d\nu > 1\}\in \mathcal{R}$ as $\proj_{ G(\mathcal{R})}\d\mu / \d\nu \in  G(\mathcal{R})$.  For the reverse inequality, by Corollary~\ref{cor:integral-proj-ineq}, we have
\[
\int_A \Bigl( \proj_{ G(\mathcal{R})} \frac{\d\mu}{\d\nu} - 1 \Bigr) \d\nu \ge \mu(A)-\nu(A),
\]
for any $A\in\mathcal{R}$. By \eqref{eq:same},
\begin{align*}
\int  \frac{1}{2}  \Bigl\lvert \proj_{ G(\mathcal{R})} \frac{\d\mu}{\d\nu}-1\Bigr\rvert \,\d\nu &=	\int_E \Bigl( \proj_{ G(\mathcal{R})} \frac{\d\mu}{\d\nu}-1 \Bigr) \d\nu \\
&=\sup_{A\in\mathcal{R}}\Bigl( \int_A \proj_{ G(\mathcal{R})} \frac{\d\mu}{\d\nu}\,\d\nu -\nu(A)\Bigr) \ge \sup_{A\in\mathcal{R}} \, \bigl(\mu(A)-\nu(A)\bigr)
\end{align*}
as required.
\end{proof}
 
%We end the section with a comment about Corollary~\ref{cor:RNP}, which is stated for the $\pi$-system $\mathcal{E}(\d\mu/\d\nu)$. We may extend it, by way of Corollary~\ref{cor:generate}, to the $\sigma$-algebra generated by $\mathcal{E}(\d\mu/\d\nu)$, i.e., $\rho_{\mathcal{R}}^{-1}(\mathcal{B})$. For any $E\in\rho_\mathcal{R}^{-1}(\mathcal{B})$,
%\begin{equation}\label{eq: RNP}
%\mu(E)=\int_E \rho_\mathcal{R}\,\d\nu.
%\end{equation}
%However, observe that equation \eqref{eq: RNP} does not recover the standard version of the Radon--Nikodym property when we replace $\mathcal{R}$ in equation \eqref{eq: RNP} with a $\sigma$-subalgebra $\mathcal{C}\subseteq \mathcal{B}$, since, in general, $\rho_\mathcal{C}^{-1}(\mathcal{B})\subsetneq \mathcal{C}$. For example, consider the case where $\mu=\nu$, In this setting $\rho_\mathcal{C}=1$ is a constant function, Consequently, we have $\rho_\mathcal{C}^{-1}(\mathcal{B})=\{\varnothing,\mathbb{R}\}$. Thus, equation \eqref{eq: RNP} reduces to a trivial identity.

\section{Properties of the $f$-divergences over rays}\label{sec:ineq}

Theorem~\ref{thm:main-theorem} serves as an indication that our notion of $f$-divergence over rays in Definition~\ref{def:Rfdiv} is the right one because it restricts to the Kolmogorov--Smirnov distance for a special choice of $f$. In this section, we will see that just about every property that we know holds for $f$-divergences also holds for $f$-divergences over $\mathcal{R}$. Henceforth we will denote the convex cone of functions convex on $[0, \infty)$ and both vanishing and strictly convex at $1$ by
\begin{equation}\label{eq:mathcalF}
\mathcal{F} \coloneqq \{f : [0,\infty) \to \mathbb{R} : f \text{ is convex, } f \text{ is strictly convex at 1,  and } f(1)=0\}.	 
\end{equation}
The strict convexity at $1$ is as in  \cite[p.~120]{Polyanskiy_Wu_2024}: For all $x,y\in [0,\infty)$ and $t \in (0,1)$ such that $t x+(1- t)y=1$, we have $t f(x)+(1- t)f(y)>0$. This condition guarantees that $\D_f(\mu\Vert\nu)=0$ if and only if $\mu=\nu$. As a precaution, in situations like the Kullback--Leibler divergence where $f(x) = x \log x$, the value $f(0)$ may only be defined in a limiting sense.

Again $\mu,\nu$ will denote Borel probability measures on $(\mathbb{R},\mathcal{B})$ throughout this section, with the implicit assumption that $\mu\ll \nu$ whenever we speak of their Radon--Nikodym derivative  $\d\mu / \d\nu$. 
\begin{theorem}\label{thm:prop}
Let $f, g \in \mathcal{F}$. Then the $f$-divergence over $\mathcal{R}$ has the following properties:
\begin{enumerate}[\upshape (i)]
\item\label{it:lin} Linearity: $\D_{\alpha f + \beta g}^\mathcal{R} (\mu \Vert \nu) = \alpha \D_f^\mathcal{R}(\mu \Vert \nu) + \beta \D_g^\mathcal{R}(\mu \Vert \nu)$ for any $\alpha, \beta\ge 0$.
\item\label{it:pos} Nonnegativity: $\D_f^\mathcal{R}(\mu \Vert \nu) \ge 0$.
\item\label{it:aff} Affine invariance: If $g(x)=f(x)+c(x-1)$, $c \in \mathbb{R}$, then $\D_f^\mathcal{R}(\mu \Vert \nu) = \D_g^\mathcal{R}(\mu \Vert \nu)$.
\item\label{it:bd} Boundedness: $\D_f^\mathcal{R}(\mu \Vert \nu)\le \D_f(\mu \Vert \nu)$.
\item\label{it:id} Identity: If $\D_f^\mathcal{R}(\mu \Vert \nu) =0 = \D_f^\mathcal{R}(\nu \Vert \mu)$, then $\mu=\nu$.
\end{enumerate}
\end{theorem}

\begin{proof} 
Item~\ref{it:lin} is routine from definition. For item~\ref{it:pos}, since $f$ is convex, applying Jensen's inequality and Theorem~\ref{thm:key-lemma} gives
\[
\int  f\Bigl(\proj_{ G(\mathcal{R})} \frac{\d\mu}{\d\nu}\Bigr)\,\d\nu \ge f\Bigl(\int  \proj_{ G(\mathcal{R})} \frac{\d\mu}{\d\nu}\,\d\nu \Bigr)	=f\Bigl(\int   \frac{\d\mu}{\d\nu}\,\d\nu \Bigr) =f(1)=0.
\]
Item~\ref{it:aff} follows from
\begin{align*}
	\D_g^\mathcal{R} (\mu \Vert \nu) &= \int  f\Bigl(\proj_{ G(\mathcal{R})}\frac{\d\mu}{\d\nu}\Bigr) + c\Bigl(\proj_{ G(\mathcal{R})} \frac{\d\mu}{\d\nu} - 1\Bigr) \,\d\nu \\
	&=\int  f\Bigl(\proj_{ G(\mathcal{R})}\frac{\d\mu}{\d\nu}\Bigr) \,\d\nu + c \int \Bigl( \proj_{ G(\mathcal{R})} \frac{\d\mu}{\d\nu} -1 \Bigr) \d\nu =\D_f^\mathcal{R}(\mu \Vert \nu),
\end{align*}
noting that the last integral vanishes. The proofs for items~\ref{it:bd} and \ref{it:id} are considerably more involved and will be deferred to Proposition~\ref{prop:cvxineq} and Corollary~\ref{cor:lowerbound} respectively.
\end{proof}

We caution the reader that special properties of a specific $f$-divergence may sometimes be lost for its counterpart over $\mathcal{R}$. For example, it is well-known that while $f$-divergence is generally not symmetric, i.e., $\D_f(\mu \Vert \nu) \ne \D_f(\nu \Vert \mu)$, the total variation distance is. This symmetry is lost for the total variation distance over $\mathcal{R}$ as
\[
\D_{\tv}^\mathcal{R}(\mu \Vert \nu)=\int \frac{1}{2}\Bigl\lvert \proj_{ G(\mathcal{R})}\frac{\d\mu}{\d\nu} -1\Bigr\rvert\,\d\nu =\sup_{A\in\mathcal{R}}\,\bigl(\mu(A)-\nu(A)\bigr);
\]
although one may recover it by symmetrizing:
\begin{equation}\label{eq:dist}
\smash[t]{\widehat{\D}}_{\tv}^\mathcal{R}(\mu \Vert \nu) \coloneqq \max\bigl\{\D_{\tv}^\mathcal{R}(\mu \Vert \nu),\D_{\tv}^\mathcal{R}(\nu \Vert \mu) \bigr\}=\sup_{A\in\mathcal{R}}\, \lvert \mu(A)-\nu(A) \rvert.
\end{equation}
Strictly speaking, it is this symmetrized version $\smash[t]{\widehat{\D}}_{\tv}^\mathcal{R}$ that should be called the Kolmogorov--Smirnov \emph{distance} \cite{kelbert2023} although we have been loosely using the term for the \emph{divergence} $\D_{\tv}^\mathcal{R}$ as well.

This is also the reason why we need both $\D_f^\mathcal{R}(\mu \Vert \nu) =0$ and $\D_f^\mathcal{R}(\nu \Vert \mu) = 0$ in Theorem~\ref{thm:prop}\ref{it:id}. It would not be true otherwise. We illustrate this in Figure~\ref{fig:three_images} with the total variation and Hellinger distances over $\mathcal{R}$. One can see that   $\D_f^\mathcal{R}(\mu \Vert  \nu) = 0$  or  $\D_f^\mathcal{R}(\nu \Vert  \mu) = 0$   does not necessarily imply  $\mu = \nu$ but that $\D_f^\mathcal{R}(\mu \Vert  \nu) = \D_f^\mathcal{R}(\nu \Vert  \mu) = 0$ does.
\begin{figure}[htp]
    \centering
        \hfill
    \begin{subfigure}[b]{0.23\textwidth}
        \centering
        \includegraphics[width=\textwidth]{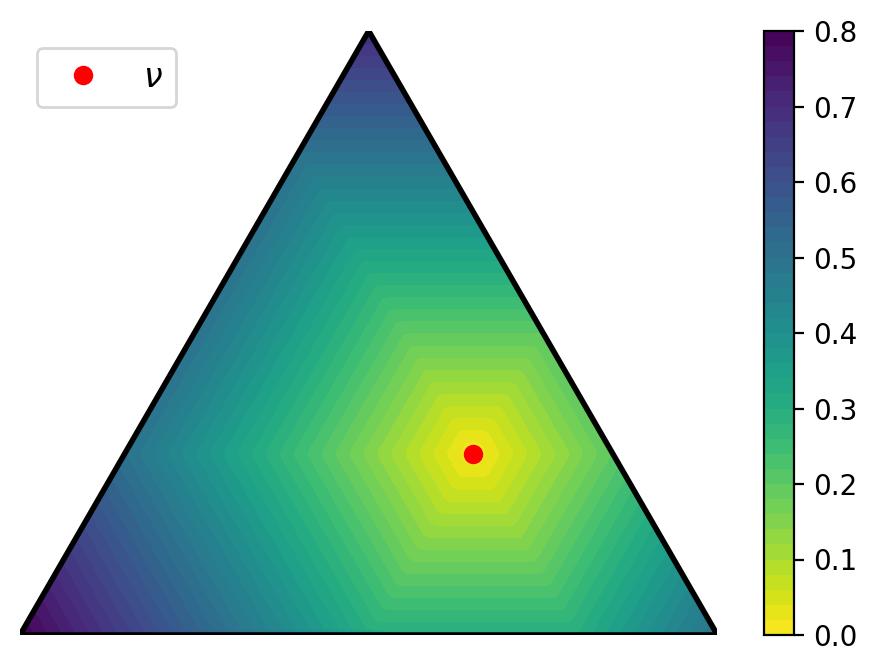}
        \caption{$\D_{\tv}(\mu \Vert \nu)$}
        \label{fig:image1}
    \end{subfigure}
    \hfill
    \begin{subfigure}[b]{0.23\textwidth}
        \centering
        \includegraphics[width=\textwidth]{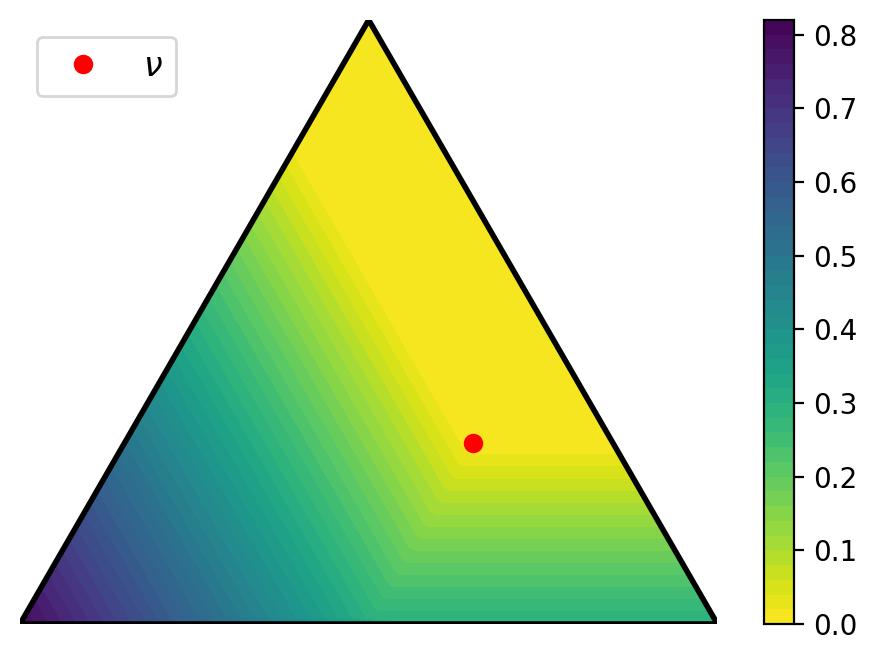}
        \caption{$\D_{\tv}^\mathcal{R}(\mu \Vert \nu)$}
        \label{fig:image2}
    \end{subfigure}
    \hfill
    \begin{subfigure}[b]{0.23\textwidth}
        \centering
        \includegraphics[width=\textwidth]{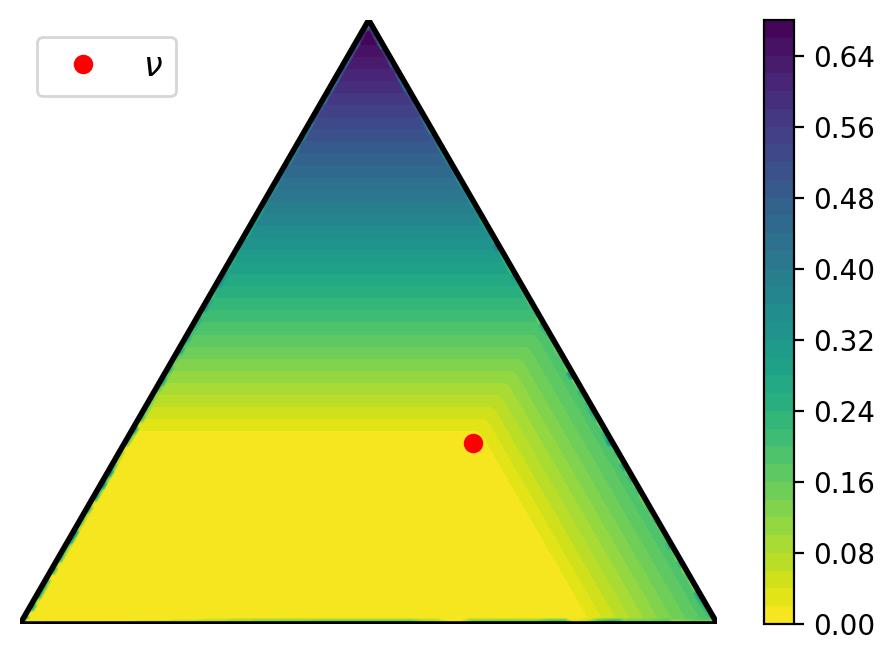}
        \caption{$\D_{\tv}^\mathcal{R}(\nu \Vert \mu)$}
        \label{fig:image3}
    \end{subfigure}
        \hfill
    
    \hfill
        \begin{subfigure}[b]{0.23\textwidth}
        \centering
        \includegraphics[width=\textwidth]{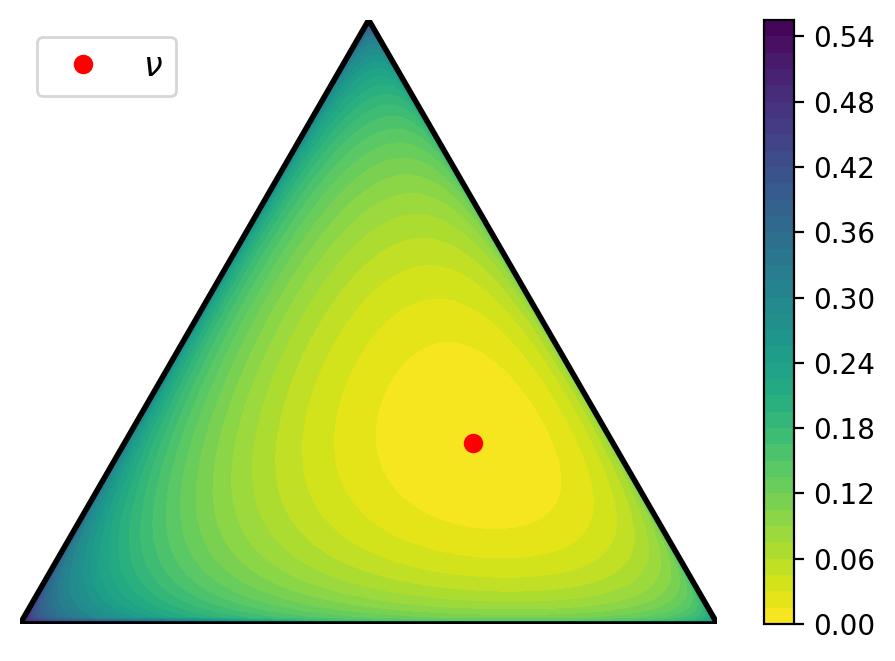}
        \caption{$\D_{\hl}(\mu \Vert \nu)$}
        \label{fig:image1}
    \end{subfigure}
    \hfill
    \begin{subfigure}[b]{0.23\textwidth}
        \centering
        \includegraphics[width=\textwidth]{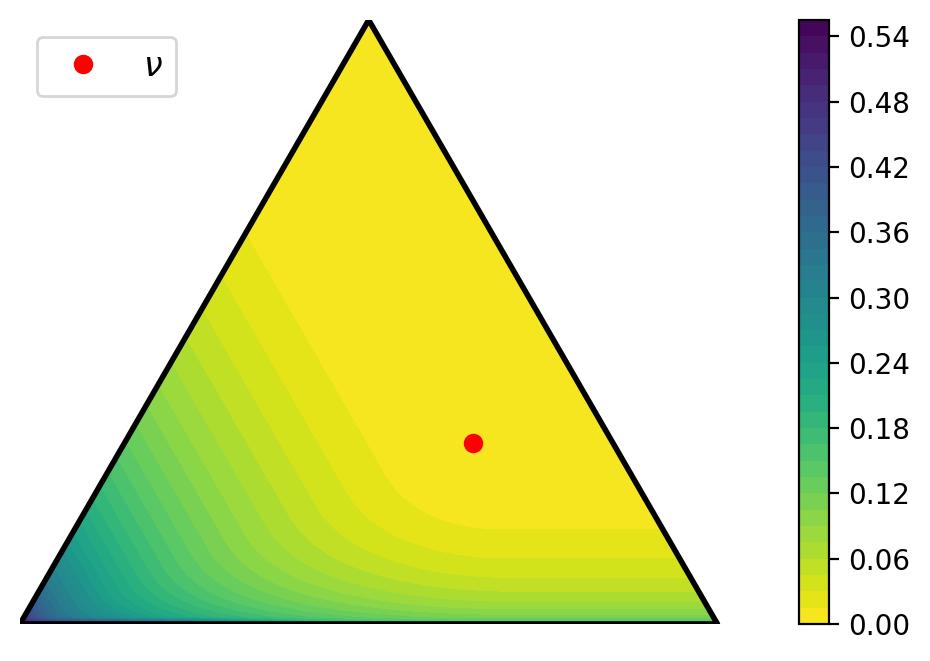}
        \caption{$\D_{\hl}^\mathcal{R}(\mu \Vert \nu)$}
        \label{fig:image2}
    \end{subfigure}
    \hfill
    \begin{subfigure}[b]{0.23\textwidth}
        \centering
        \includegraphics[width=\textwidth]{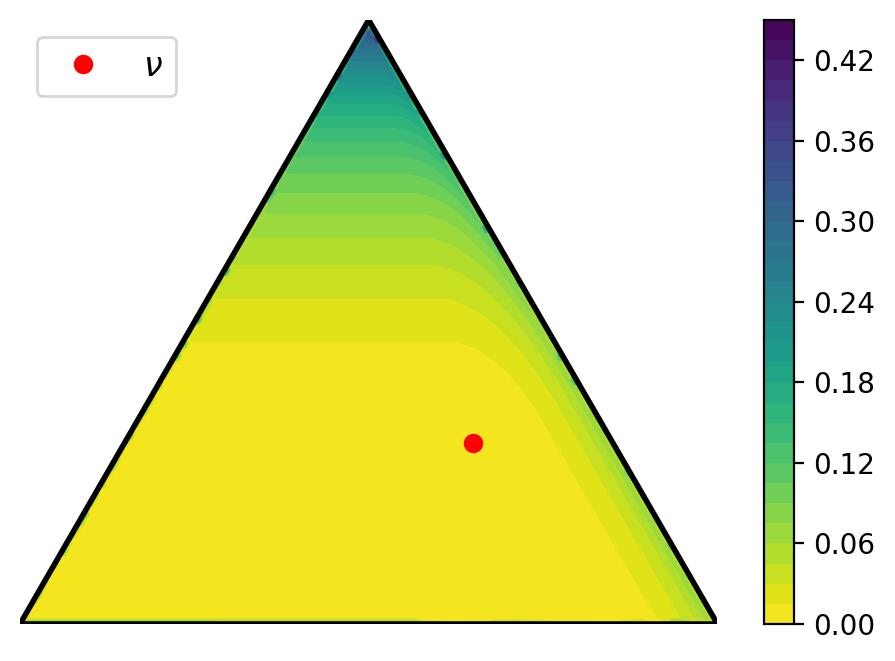}
        \caption{$\D_{\hl}^\mathcal{R}(\nu \Vert \mu)$}
        \label{fig:image3}
    \end{subfigure}
        \hfill
    \caption{40 Level curves of total variation $\D_{\tv}$, total variation over rays $\D_{\tv}^\mathcal{R}$, Hellinger distance $\D_{\hl}$ and Hellinger distance over rays $\D_{\hl}^\mathcal{R}$ for fixed $\nu=[0.2,0.5,0.3]$ as $\mu$ ranges over the simplex of distributions on a three-element set.}
    \label{fig:three_images}
\end{figure}

Evidently, the projection operator $\proj_{ G(\mathcal{R})}$ plays a key role in our definition of $f$-divergence over $\mathcal{R}$. The next two results show how it interacts with nondecreasing functions and with convex functions.
\begin{lemma}\label{lem:integral-proj-eq} 
Let $f : \mathbb{R} \to \mathbb{R}$ be nondecreasing. If $g\in L^2_\p (\nu)$ is almost surely continuous, then
\[
\int  (\proj_{ G(\mathcal{R})} g) \cdot (f\circ \proj_{ G(\mathcal{R})} g) \,\d\nu = \int   g \cdot (f\circ \proj_{ G(\mathcal{R})} g)\,\d\nu.
\]
\end{lemma}
\begin{proof} 
Let $h=\sum_{i=1}^{n} \alpha_i\1_{A_i}$ and $\proj_{ G(\mathcal{R})} h = \sum_{i=1}^{n} \beta_i \1_{A_i}$. Suppose $\proj_{ G(\mathcal{R})} h$ takes $m$ distinct values $\tau_1 >  \dots > \tau_m$ where $m \le n$. For each $k=1,\dots,m$, set
\[
E_k \coloneqq \bigcup_{i=1}^k\{\proj_{ G(\mathcal{R})} h = \tau_i\}\in \mathcal{E}(h).
\]
Let $\zeta_m \coloneqq f(\tau_m )$ and  $\zeta_{m-j} \coloneqq f(\tau_{m-j}) - \zeta_{m-j+1}-\zeta_{m-j+2}-\dots  -\zeta_m $, $j=1, \dots,m-1$.
Then
\begin{align*}
\int  (\proj_{ G(\mathcal{R})} h) \cdot (f\circ \proj_{ G(\mathcal{R})} h) \,\d\nu 
&= \int  (\proj_{ G(\mathcal{R})} h) \cdot \Bigl(\sum_{i=1}^{n} f(\beta_i)\1_{A_i}\Bigr) \,\d\nu \\
&= \int  (\proj_{ G(\mathcal{R})} h) \cdot \sum_{i=1}^{m} \zeta_i \1_{E_i} \,\d\nu  \\
&=\sum_{i=1}^{m} \zeta_i \int_{E_i} \proj_{ G(\mathcal{R})} h \,\d\nu =	\sum_{i=1}^{m} \zeta_i \int_{E_i}  h \,\d\nu \\
&=\int  h \cdot \sum_{i=1}^{m} \zeta_i \1_{E_i} \,\d\nu =\int  h \cdot (f\circ \proj_{ G(\mathcal{R})} h)\,\d\nu.
\end{align*}
Now let $h_n$ be a nondecreasing sequence of nonnegative ordered simple function that converges to $g$ almost surely and in $L^2$-norm. By passing through a subsequence if necessary, $\proj_{ G(\mathcal{R})} h_n$ is a nonincreasing sequence that converges to $\proj_{ G(\mathcal{R})} g$ almost surely. Since $f$ is  nondecreasing,
\[
f\circ \proj_{ G(\mathcal{R})} h_n\uparrow f\circ \proj_{ G(\mathcal{R})} g
\]
almost surely.  By the monotone convergence theorem,
\begin{multline*}
\int  (\proj_{ G(\mathcal{R})} g) \cdot (f\circ \proj_{ G(\mathcal{R})} g) \,\d\nu =\lim_{n\to\infty} \int  (\proj_{ G(\mathcal{R})} h_n) \cdot (f\circ \proj_{ G(\mathcal{R})} h_n) \,\d\nu \\
	=\lim_{n\to\infty} \int   h_n \cdot (f\circ \proj_{ G(\mathcal{R})} h_n) \,\d\nu = \int   g \cdot (f\circ \proj_{ G(\mathcal{R})} g)\,\d\nu. \qedhere
\end{multline*}
\end{proof}

\begin{proposition}\label{prop:cvxineq} 
Let $f : \mathbb{R} \to \mathbb{R}$ be convex. If $g\in L^2_\p (\nu)$ is almost surely continuous, then
\[
\int  f \circ { \proj_{ G(\mathcal{R})} g} \,\d\nu \le \int  f\circ g \,\d\nu.
\]
\end{proposition}

\begin{proof} 
By convexity,
\[
f\bigl(\proj_{ G(\mathcal{R})} g(x) \bigr)-f\bigl(g(x)\bigr)\le f'\bigl(\proj_{ G(\mathcal{R})} g(x) \bigr)\cdot \bigl(\proj_{ G(\mathcal{R})} g(x)-g(x)\bigr)
\]
almost surely. Integrating with respect to $\nu$ and applying Lemma~\ref{lem:integral-proj-eq} to $f'$, we get
\[
	\int f \circ { \proj_{ G(\mathcal{R})} g} - f\circ g\,\d\nu \le \int f'(\proj_{ G(\mathcal{R})} g)\cdot (\proj_{ G(\mathcal{R})} g- g)\,\d\nu =0. \qedhere
\]
\end{proof}
Theorem~\ref{thm:prop}\ref{it:bd} follows from Proposition~\ref{prop:cvxineq} since any $f \in \mathcal{F}$ is convex and so
\[
\D_f^\mathcal{R}(\nu \Vert \mu) = \int  f\Bigl(\proj_{ G(\mathcal{R})} \frac{\d\mu}{\d\nu}\Bigr)\,\d\nu \le \int  f\Bigl(\frac{\d\mu}{\d\nu}\Bigr)\,\d\nu =  \D_f(\nu \Vert \mu).	
\]
It remains to establish Theorem~\ref{thm:prop}\ref{it:id}, whose proof will take four intermediate results: Propositions~\ref{prop:Rfdiv-in-fdiv} and \ref{prop:fdiv-in-Rfdiv}, Corollaries~\ref{cor:Rfdiv-iff-fdiv} and \ref{cor:lowerbound}, each of independent interest. We begin by showing how we may convert $f$-divergence over rays to standard $f$-divergence.
\begin{proposition}\label{prop:Rfdiv-in-fdiv}  
For any  Borel probability measures $\mu$ and $\nu$, there exists a Borel probability measure $\zeta$ such that for any $f\in\mathcal{F}$,
\[
\D_f^\mathcal{R}(\mu \Vert \nu) = \D_f(\zeta \Vert \nu).
\]
\end{proposition}

\begin{proof}
We define  $\zeta$ by
\[
\zeta(B) \coloneqq \int_B \proj_{ G(\mathcal{R})}\frac{\d\mu}{\d\nu} \,\d\nu
\]
for any $B\in\mathcal{B}$.
By Corollary~\ref{cor:RNP}, $\zeta$ is a probability measure absolutely continuous with respect to $\nu$. Hence
\[
\frac{\d\zeta}{\d\nu}=\proj_{ G(\mathcal{R})} \frac{\d\mu}{\d\nu}
\]
and so
\[
\D_f^\mathcal{R}(\mu \Vert \nu) = \int  f\Bigl(\proj_{ G(\mathcal{R})}\frac{\d\mu}{\d\nu}\Bigr)\,\d\nu = \int  f\Bigl(\frac{\d\zeta}{\d\nu}\Bigr)\,\d\nu =\D_f(\zeta \Vert \nu). \qedhere
\] 
\end{proof}
Next we do the reverse: converting standard $f$-divergence to $f$-divergence over rays.
\begin{proposition}\label{prop:fdiv-in-Rfdiv} 
For any  Borel probability measures $\mu$ and $\nu$, there exist Borel probability measures $\eta, \tau$ such that for any $f\in\mathcal{F}$,
\[
\D_f(\mu \Vert \nu) = \D_f^\mathcal{R}(\eta \Vert \tau).
\]	
\end{proposition}
\begin{proof} 
The distribution function\footnote{In this article, we use the \emph{decumulative} distribution function instead of the more common cumulative distribution function. They are equivalent --- we could have used the latter but it introduces more notational clutter. Another advantage is consistency with our discussion of Choquet integrals in Section~\ref{sec:cho}.} $m : [0,\infty) \to [0,1]$ of $\d\mu / \d\nu$ is given by
\[
m(\alpha) \coloneqq \nu \Bigl(\Bigl\{ x \in \Omega : \frac{\d\mu}{\d\nu}(x) > \alpha \Bigr\} \Bigr)
\]
for any $\alpha \in [0,\infty)$. Since $m$ may not be injective, we define its generalized inverse $\rho : [0,1] \to [0,\infty)$,
\[
\rho (u) \coloneqq \inf \{ \alpha \in [0,\infty) : m(\alpha)\leq u\}
\]
for  any $u\in[0,1]$.
Thus $\rho$ may be viewed a decreasing rearrangement of $\d\mu / \d\nu$ on $[0,1]$ satisfying
\[
\nu \Bigl(\Bigl\{x \in \Omega : \frac{\d\mu}{\d\nu}(x)>\alpha \Bigr\}\Bigr) = \lambda\bigl(\{u\in [0,1] :\rho(u)>\alpha\}\bigr),
\]
with $\lambda$ the Lebesgue measure on $[0,1]$ \cite[p.~199]{folland1999}. Let $T: [0,1] \to \mathbb{R}$ be a strictly increasing measure-preserving transformation. Define the probability measures
\[
\tau(A) \coloneqq \lambda(T^{-1}(A)),\qquad
\eta(A) \coloneqq \int_{T^{-1}(A)} \rho \,\d\lambda
\]
for any $A \in \mathcal{B}$. Since $\rho$ is nonincreasing and $T^{-1}$ is increasing, the Radon--Nikodym derivative
\[
\frac{\d\eta}{\d\tau} = \rho\circ T^{-1}
\]
is also nonincreasing on $\mathbb{R}$, so $\d\eta/\d\tau = \proj_{G(\mathcal{R})} \d\eta/\d\tau$. Hence, for any $f\in\mathcal{F}$,
\[
\D_f(\mu \Vert \nu) = \int f\Bigl(\frac{\d\mu}{\d\nu}\Bigr)\,\d\nu =\int_{[0,1]} f(\rho)\,\d\lambda 
= \int f(\rho\circ T^{-1})\,\d\tau = \int  f\Bigl(\proj_{ G(\mathcal{R})}\frac{\d\eta}{\d\tau}\Bigr)\,\d\tau = \D_f^{\mathcal{R}} (\eta\Vert\tau)
\]
as required.
\end{proof}
Collectively, Propositions~\ref{prop:Rfdiv-in-fdiv} and \ref{prop:fdiv-in-Rfdiv} lead to the following conclusion: An inequality relation between two divergences hold if and only if that same relation hold for their counterparts over rays.
\begin{corollary}[Preservation of inequality relations]\label{cor:Rfdiv-iff-fdiv} 
Let $f, g\in\mathcal{F}$ and $\psi, \varphi : [0,\infty) \to \mathbb{R}$. Then
\[
\psi(\D_f(\mu \Vert \nu))\le \varphi(\D_g(\mu \Vert \nu))
\]
holds for all Borel probability measures $\mu$ and $\nu$ if and only if
\[
\psi(\D^\mathcal{R}_f(\mu \Vert \nu))\le \varphi(\D^\mathcal{R}_g(\mu \Vert \nu))
\]
holds for all Borel probability measures $\mu$ and $\nu$.
\end{corollary}

\begin{proof} 
Suppose $\psi(\D_f(\mu \Vert \nu))\le \varphi(\D_g(\mu \Vert \nu))$ for all $\mu$ and $\nu$. By Proposition~\ref{prop:Rfdiv-in-fdiv}, there exists $\zeta$ such that 
\[
\psi(\D_f^\mathcal{R}(\mu \Vert \nu)) = \psi(\D_f(\zeta \Vert \nu))\le \varphi(\D_g(\zeta \Vert \nu))=\varphi(\D_g^\mathcal{R}(\mu \Vert \nu)).
\]
Conversely, suppose $\psi(\D_f^\mathcal{R}(\mu \Vert \nu))\le \varphi(\D^\mathcal{R}_g(\mu \Vert \nu))$ for all $\mu$ and $\nu$. By Proposition~\ref{prop:fdiv-in-Rfdiv}, there exist $\eta$ and $\tau$ such that  
\[
\psi(\D_f(\mu \Vert \nu))= \psi(\D_f^\mathcal{R}(\eta \Vert \tau))\le \varphi(\D_g^\mathcal{R}(\eta \Vert \tau)) = \varphi(\D_g(\mu \Vert \nu)).\qedhere
\]
\end{proof}

As most readers familiar with $f$-divergences would suspect, the total variation distance plays a somewhat special role among all $f$-divergences. One striking property \cite{brocker2009}  is that it serves as a kind of universal lower bound for all $f$-divergence, in the sense that
\[
\bar{f}\bigl(\D_{\tv}(\mu \Vert \nu) \bigr) \leq \D_f(\mu \Vert \nu),
\]
where $\bar{f}(x) \coloneqq f(1+x) + f(1-x)$ is the symmetrized function of $f$. Note that $\D_{\tv}(\mu \Vert \nu)$ takes value in $[0,1]$, so the $f(1-x)$ term is always evaluated within the domain of $f \in \mathcal{F}$. A consequence of our definition of $\mathcal{F}$ in \eqref{eq:mathcalF} is that $\bar{f}$ is increasing and $\bar{f}(x)>0$ when $x>0$. So by Corollary~\ref{cor:Rfdiv-iff-fdiv}, the same result holds verbatim for divergences over rays.

\begin{corollary}[Total variation over rays as universal lower bound]\label{cor:lowerbound} 
Let $f \in \mathcal{F}$ and set $\bar{f}(x) \coloneqq f(1+x) + f(1-x)$ for $x \in [0,1]$. Then
\[
\bar{f}\bigl(\D^\mathcal{R}_{\tv}(\mu \Vert \nu)\bigr) \leq \D^\mathcal{R}_f(\mu \Vert \nu).
\]
\end{corollary}

Theorem~\ref{thm:prop}\ref{it:id} now follows directly from Corollary~\ref{cor:lowerbound}: If $\D^\mathcal{R}_f(\mu \Vert \nu) = 0$, then $\bar{f}\bigl(\D^\mathcal{R}_{\tv}(\mu \Vert \nu)\bigr) = 0$. So $\D^\mathcal{R}_{\tv}(\mu \Vert \nu) = 0$ as $\bar{f}$ is strictly increasing and $\bar{f}(0) = 0$. Since we have also assumed that $\D^\mathcal{R}_f(\nu \Vert \mu) = 0$, we get $\D^\mathcal{R}_{\tv}(\nu \Vert \mu) = 0$, and thus $\smash[t]{\widehat{\D}}_{\tv}^\mathcal{R}(\nu \Vert \mu) = 0$. By \eqref{eq:dist}, we deduce that $\mu = \nu$ over $\mathcal{R}$. But $\mathcal{R}$ is a $\pi$-system that generates $\mathcal{B}$. So Lemma~\ref{lem:pi-sys} assures that we have $\mu = \nu$.

We end this section with another demonstration of the utility of Corollary~\ref{cor:Rfdiv-iff-fdiv}: It allows us to extend essentially all known relations between divergences to their counterparts over rays.
\begin{theorem}[Relations between $f$-divergences over rays]\label{thm:rel}
The following inequalities hold verbatim when each $\D_f$ is replaced by $\D_f^\mathcal{R}$.
\begin{enumerate}[\upshape (a)]
\item\label{tvh} Total variation and Hellinger:
\begin{gather*}
\frac{1}{2} \D_{\hl}^2(\mu \Vert \nu) \le \D_{\tv}(\mu \Vert \nu) \le \D_{\hl}(\mu \Vert \nu) \sqrt{1-\frac{\smash[t]{\D_{\hl}^2}(\mu \Vert \nu)}{4}} \le 1, \\
\D_{\tv}(\mu \Vert \nu) \le \sqrt{-2 \log \Bigl(1-\frac{\smash[t]{\D_{\hl}^2}(\mu \Vert \nu)}{2}\Bigr)}.
\end{gather*}

\item\label{kltv} Kullback--Leibler and total variation:
\begin{gather*}
\D_{\tv}^2(\mu \Vert \nu) \le \frac{1}{2}\D_{\kl}(\mu \Vert \nu), \\
\log\Bigl( \frac{1+\D_{\tv}(\mu \Vert \nu)}{1-\D_{\tv}(\mu \Vert \nu)} \Bigr) -\frac{2 \D_{\tv}(\mu \Vert \nu)}{1+\D_{\tv}(\mu \Vert \nu)}  \le \D_{\kl}(\mu \Vert \nu), \\
\D_{\kl}(\mu \Vert \nu) \le \log \Bigl(1+\frac{1}{2\nu_{\min }} \D_{\tv}(\mu \Vert \nu)^2\Bigr) \le \frac{1}{2\nu_{\min}} \D_{\tv}(\mu \Vert \nu)^2,
\end{gather*}
where $\nu_{\min} \coloneqq \min _x \nu(x)$.
These are known respectively as Pinsker, strong Pinsker, and reverse Pinsker inequality.

\item\label{tvchi} Total variation and $\chi^2$:
\begin{gather*}
\D_{\tv}^2(\mu \Vert \nu) \le \frac{1}{4}\D_{\chi^2}(\mu, \nu), \qquad \D_{\tv}(\mu \Vert \nu) \le \max \biggl\{\frac{1}{2}, \frac{\D_{\chi^2}(\mu \Vert \nu)}{1+\D_{\chi^2}(\mu, \nu)}\biggr\}, \\
4 \D_{\tv}^2(\mu \Vert \nu)\le  \varphi (\D_{\tv}(\mu \Vert \nu)) \le \D_{\chi^2}(\mu \Vert \nu),
\end{gather*}
where
\[
\varphi (t)= \begin{cases}4 t^2 & t \le 1/2, \\ t/(1-t) & t \ge 1/2.
\end{cases}
\]

\item\label{klh} Kullback--Leibler and Hellinger:
\begin{gather*}
\D_{\hl}^2(\mu \Vert \nu)\le 2 \log \frac{2}{2-\D_{\hl}^2(\mu \Vert \nu)}  \le \D_{\kl}(\mu \Vert \nu), \\
\D_{\kl}(\mu \Vert \nu) \le \frac{\log \bigl(1/\nu_{\min}-1\bigr)}{1-2\nu_{\min}}\bigl(1 - ( 1-\D_{\hl}^2(\mu \Vert \nu))^2\bigr).
\end{gather*}

\item\label{klchi} Kullback--Leibler and $\chi^2$:
\[
\D_{\kl}(\mu \Vert \nu)\le \log \bigl(1+\D_{\chi^2}(\mu \Vert \nu)\bigr) \le  \D_{\chi^2}(\mu \Vert \nu).
\]
\item\label{lch}  Le~Cam and Hellinger:
\[
\frac{1}{2} \D_{\hl}^2(\mu \Vert \nu) \le \D_{\lc}(\mu \Vert \nu) \le \D_{\hl}^2(\mu \Vert \nu).
\]
\item\label{lcjs}  Le~Cam and Jensen--Shannon:
\[
\D_{\lc}(\mu \Vert \nu)  \le \D_{\js}(\mu \Vert \nu) \le 2 \log 2\cdot\D_{\lc}(\mu \Vert \nu).
\]
\end{enumerate}
\end{theorem}
\begin{proof}
These are well-known inequalities between the respective $f$-divergences: the two inequalities in \ref{tvh} may be found in \cite[p.~124]{Polyanskiy_Wu_2024} and \cite[Theorem~12]{Gilardoni2010} respectively; the three inequalities in \ref{kltv} in \cite[p.~88]{Tsybakov2009}, \cite{Vajda1970}, and  \cite[Theorem~28]{Sason2016} respectively; those in \ref{tvchi}, \ref{klh}, and \ref{klchi} may be found in \cite[p.~133]{Polyanskiy_Wu_2024}; \ref{lch} in \cite[p.~48]{LeCam1986}; \ref{lcjs} in \cite[Theorem~3.2]{Topsoe2000}. It follows from Corollary~\ref{cor:Rfdiv-iff-fdiv} that these relations also hold for the corresponding $f$-divergences over $\mathcal{R}$.
\end{proof}
Ultimately, a definition is justified by the richness of results that can be proved with it. We view the results in this section as evidence that the $f$-divergence over rays as defined in Definition~\ref{def:Rfdiv} is the ``right'' one.  In this regard, the next section will supply another major piece of evidence.

\section{Glivenko--Cantelli theorem for $f$-divergences over rays}\label{sec:GC}

Put in our context, the original Glivenko--Cantelli theorem in Theorem~\ref{thm:GC} states that empirical distributions converge to the target distribution almost surely in total variation distance over $\mathcal{R}$, as defined in Definition~\ref{def:Rfdiv}. Here we generalize this result to all $f$-divergences over $\mathcal{R}$.

As in the case of total variation distance, this convergence does not always hold if we replace $\mathcal{R}$ with $\mathcal{B}$ or some other subsets of $\mathcal{B}$. Examples~\ref{ex:tv} and \ref{ex:gc}, where $\limsup_{n\to\infty} \D_{\tv}(\nu_n \Vert \nu)>0$, can be readily extended to any $f\in\mathcal{F}$ by way of Corollary~\ref{cor:lowerbound}:
\[
0<\bar{f}\bigl(\limsup_{n\to\infty} \D_{\tv}(\nu_n \Vert \nu)\bigr) = \limsup_{n\to\infty} \bar{f}\bigl(\D_{\tv}(\nu_n \Vert \nu)\bigr) \le \D_f(\nu_n \Vert \nu),
\]
i.e., the $f$-divergence between the empirical distribution and the target does not in general converge to $0$ almost surely. Our main result is that it holds for $f$-divergence over rays: For any $f\in\mathcal{F}$, Theorems~\ref{thm:GC1} and \ref{thm:GC2} collectively show that almost surely
\begin{equation}\label{eq:main}
\lim_{n\to\infty} \D_f^\mathcal{R}(\nu_n \Vert \nu) =0 = \lim_{n\to\infty} \D_f^\mathcal{R}(\nu \Vert \nu_n).
\end{equation}
Note that an $f$-divergence is generally not symmetric in its arguments so we need to establish both equalities.

We will examine more carefully why a sequence of empirical measures can converge on $\mathcal{R}$ but fails to converge on $\mathcal{B}$. It turns out that a necessary condition is that its corresponding sequence of Radon--Nikodym derivatives does not converge to $1$, but the sequence of their projections onto $G(\mathcal{R})$ does.
\begin{proposition}\label{prop:RN}
Let $\nu$ be a Borel probability measure and $\nu_n$ be the corresponding empirical measure, $n \in \mathbb{N}$.
Suppose the Radon--Nikodym derivative $\d\nu_n / \d\nu$ exists for all $n\in\mathbb{N}$ and are uniformly bounded. If
\[
\sup_{A\in\mathcal{B}}\bigl(\nu_n(A)-\nu(A)\bigr)\not\to 0 \quad\text{and}\quad \sup_{A\in\mathcal{R}}\bigl(\nu_n(A)-\nu(A)\bigr)\to 0,
\]
then
\[
 \frac{\d\nu_n}{\d\nu}\not\to 1 \quad\text{and}\quad \proj_{ G(\mathcal{R})}  \frac{\d\nu_n}{\d\nu} \to 1,
\]
where convergence are all in an almost sure sense.
\end{proposition}
\begin{proof}
Suppose $\lim_{n\to\infty} \d\nu_n / \d\nu = 1$. Then
\begin{align*}
	0 = \int  \frac{1}{2}  \biggl\lvert \lim_{n\to\infty} \frac{\d\nu_n}{\d\nu} - 1\biggr\rvert \,\d\nu 
\ge \limsup_{n\to\infty} \int  \frac{1}{2} \biggl\lvert \frac{\d\nu_n}{\d\nu} - 1\biggr\rvert \,\d\nu = \limsup_{n\to\infty} \D_{\tv}(\nu_n \Vert \nu) \ge 0.
\end{align*}
So $\lim_{n\to\infty} \D_{\tv}(\nu_n \Vert \nu) = 0$, a contradiction, and thus $\d\nu_n/\d\nu \not\to 1$.
By assumption, $\d\nu_n / \d\nu\le c$ for all $n \in \mathbb{N}$ and for some constant $c > 0$. It then follows from Proposition~\ref{prop:proj-monotone} that 
\[
\proj_{ G(\mathcal{R})} \frac{\d\nu_n}{\d\nu}\le \proj_{ G(\mathcal{R})} c =c
\]
for all $n$, i.e., the sequence $(\proj_{ G(\mathcal{R})} \d\nu_n / \d\nu)_n$ is also uniformly bounded.  Since $\dim_{\vc}(\mathcal{R})<\infty$, it follows from the Vapnik--Chervonenkis theorem (see p.~\pageref{pg:VC} for a statement) that
\[
\lim_{n\to\infty}\D^\mathcal{R}_{\tv}(\nu_n \Vert \nu)=\lim_{n\to\infty} \sup_{A\in\mathcal{R}}\, \bigl(\nu_n(A)-\nu(A)\bigr)=0
\]
almost surely; taken together with the bounded convergence theorem,
\begin{align*}
	0 &= \lim_{n\to\infty} \sup_{A\in\mathcal{R}}\, \bigl(\nu_n(A)-\nu(A)\bigr)= \lim_{n\to\infty} \int \frac{1}{2} \Bigl\lvert \proj_{ G(\mathcal{R})} \frac{\d\nu_n}{\d\nu}-1\Bigr\rvert\,\d\nu \\
	&=\int \frac{1}{2}\Bigl\lvert \lim_{n\to\infty}\proj_{ G(\mathcal{R})} \frac{\d\nu_n}{\d\nu}-1\Bigr\rvert\,\d\nu \ge 0.
\end{align*}
Hence the last integrand must converge to zero, or, equivalently,
\begin{equation}\label{eq:RN1}
 \lim_{n\to\infty}\proj_{ G(\mathcal{R})} \frac{\d\nu_n }{ \d\nu}  = 1. \qedhere
\end{equation}
\end{proof}

With the observation in Proposition~\ref{prop:RN}, we obtain the left half of \eqref{eq:main}.
\begin{theorem}[Glivenko--Cantelli theorem for $f$-divergence I]\label{thm:GC1}
Let $\mathcal{R}$ be the class of rays, $f\in\mathcal{F}$ as in \eqref{eq:mathcalF}, $\nu$ a Borel probability measure, and $\nu_n$ the corresponding empirical measure, $n \in \mathbb{N}$. Then almost surely
\[
\lim_{n\to\infty} \D_f^\mathcal{R}(\nu_n \Vert \nu)  = 0.
\]
\end{theorem}
\begin{proof} 
By \eqref{eq:RN1} and the reverse Fatou lemma, 
\begin{align*}
0 &=  \int  \Bigl\lvert f\Bigl(\lim_{n\to\infty} \proj_{ G(\mathcal{R})} \frac{\d\nu_n}{\d\nu}\Bigr)\Bigr\rvert\,\d\nu=\int  \lim_{n\to\infty}\Bigl\lvert f\Bigl( \proj_{ G(\mathcal{R})} \frac{\d\nu_n}{\d\nu}\Bigr)\Bigr\rvert\,\d\nu \\
&\ge \limsup_{n\to\infty} \int  \Bigl\lvert f\Bigl( \proj_{ G(\mathcal{R})} \frac{\d\nu_n}{\d\nu}\Bigr)\Bigr\rvert\,\d\nu  \ge \limsup_{n\to\infty} \Bigl\lvert \int  f\Bigl( \proj_{ G(\mathcal{R})} \frac{\d\nu_n}{\d\nu}\Bigr)\,\d\nu\Bigr\rvert\ge 0.
\end{align*}
Hence
\[
\lim_{n\to\infty} \int  f\Bigl(\proj_{ G(\mathcal{R})} \frac{\d\nu_n}{\d\nu}\Bigr)\,\d\nu = \lim_{n\to\infty} D^\mathcal{R}_f(\nu_n \Vert \nu)=0. \qedhere
\]
\end{proof}
The following ``reciprocal'' of \eqref{eq:RN1} will lead us to the right half of \eqref{eq:main}.
\begin{proposition}\label{prop:RN2}
Let $\nu$ be a Borel probability measure and $\nu_n$ be the corresponding empirical measure, $n \in \mathbb{N}$.
Suppose the Radon--Nikodym derivative $\d\nu_n / \d\nu$ exists for all $n\in\mathbb{N}$ and are uniformly bounded. Then almost surely
\[
\lim_{n\to\infty} \proj_{ G(\mathcal{R})} \frac{\d\nu}{\d\nu_n} = 1.
\]
\end{proposition}
\begin{proof} 
Let $x_n \coloneqq \lvert\proj_{ G(\mathcal{R})} (\d\nu / \d\nu_n)-1\rvert /2$. Since $\dim_{\vc}(\mathcal{R})<\infty$, it follows from the Vapnik--Chervonenkis theorem that
\[
\sup_{A\in\mathcal{R}} \bigl( \nu(A)-\nu_n(A)\bigr) =
\lim_{n\to\infty} \int  x_n \,\d\nu_n = 0.
\]
By assumption, $\d\nu / \d\nu_n<c$  for all $n \in \mathbb{N}$ and for some constant $c > 0$. If $\lim_{n\to\infty} x_n\ne 0$, then there exists $\varepsilon>0$ and  subsequences $(x_{n_k})$ and $(\nu_{n_k})$ such that 
\[
\varepsilon < \int  x_{n_k}\,\d\nu = \int  x_{n_k}\frac{\d\nu}{\d\nu_{n_k}}\,\d\nu_{n_k} \le c \int  x_{n_k}\,\d\nu_{n_k}\to 0,	
\]
a contradiction, and thus $\lim_{n\to\infty} x_n=0$. 
\end{proof}
\begin{theorem}[Glivenko--Cantelli theorem for $f$-divergence II]\label{thm:GC2}
Let $\mathcal{R}$ be the class of rays, $f\in\mathcal{F}$ as in \eqref{eq:mathcalF}, $\nu$ a Borel probability measure, and $\nu_n$ the corresponding empirical measure, $n \in \mathbb{N}$. Then almost surely
\[
\lim_{n\to\infty} \D_f^\mathcal{R}(\nu \Vert \nu_n)  = 0.
\]
\end{theorem}
\begin{proof} 
Let $y_n \coloneqq \lvert f(\proj_{ G(\mathcal{R})} \d\nu / \d\nu_n)\rvert$. Since $f$ is continuous at $1$,
\[
\lim_{n\to\infty} y_n = \Bigl\lvert f\Bigl(\lim_{n\to\infty} \proj_{ G(\mathcal{R})} \frac{\d\nu}{\d\nu_n}\Bigr)\Bigr\rvert=0.
\] 
So
\[
\Bigl\lvert \int  f\Bigl(\proj_{ G(\mathcal{R})} \frac{\d\nu}{\d\nu_n}\Bigr) \,\d\nu_n\Bigr\rvert \le \int \Bigl\lvert f\Bigl(\proj_{ G(\mathcal{R})} \frac{\d\nu}{\d\nu_n}\Bigr)\Bigr\rvert\,\d\nu_n  = \nu_n y_n.
\]
If $\nu_n y_n \not\to 0$, then there exists $\varepsilon > 0$ and a subsequence $(y_{n_k})$ such that
\[
\varepsilon <\lvert \nu_{n_k} y_{n_k}\rvert \le \lVert y_{n_k}\rVert\to 0,	
\]
a contradiction, and thus $\lim_{n\to\infty} \nu_ny_n = 0$, giving us the required result.
\end{proof}

The following is an immediate consequence of Theorems~\ref{thm:GC1} and \ref{thm:GC2}.
\begin{corollary}[Glivenko--Cantelli theorem for $f$-divergence III]\label{cor:GC3}
Let $\mathcal{R}$ be the class of rays, $f\in\mathcal{F}$ as in \eqref{eq:mathcalF}, $\nu$ a Borel probability measure, and $\nu_n$ the corresponding empirical measure, $n \in \mathbb{N}$. Then almost surely
\[
\lim_{n\to\infty}  \max\bigl\{\D_f^\mathcal{R}(\nu \Vert \nu_n), \D_f^\mathcal{R}(\nu_n \Vert \nu)\bigr\} = 0.
\]
\end{corollary}
One advantage afforded by the version in Corollary~\ref{cor:GC3} is that the expression
\[
\smash[t]{\widehat{\D}}_f^\mathcal{R}(\mu \Vert \nu) \coloneqq \max\bigl\{\D_f^\mathcal{R}(\mu \Vert \nu),\D_f^\mathcal{R}(\nu \Vert \mu) \bigr\}
\]
defines a distance (and not just a divergence) between probability measures.  For $f(t) = \lvert t - 1 \rvert/2$, this is exactly the Kolmogorov--Smirnov distance in \eqref{eq:dist}. We may of course also replace the expression with an $L^p$ variant for any $p \in [1, \infty)$.

\section{Vapnik--Chervonenkis theory for $f$-divergence}\label{sec:VC}

A natural follow-up question from the last section is: Given that we have a Glivenko--Cantelli theorem for $f$-divergence, is a Vapnik--Chervonenkis theory for $f$-divergence within reach? A first goal of such a theory would be to extend the work in this article from the class of rays $\mathcal{R}$ to more general classes $\mathcal{C} \subseteq \mathcal{B}$. Surprisingly, we found that the original Vapnik--Chervonenkis theorem \cite{VapnikChervonenkis1971} provides a key.

Replace $\mathcal{R}$ by a class $\mathcal{C}\subseteq \mathcal{B}$ in the definition of $G(\mathcal{R})$ in \eqref{eq:GR}. As Section~\ref{sec:Rfdiv} shows, if $G(\mathcal{C})$ satisfies the conclusions of Proposition~\ref{prop:FR} and  Theorem~\ref{thm:main-theorem} with $\mathcal{C}$ in place of $\mathcal{R}$, then Definition~\ref{def:Rfdiv} with $\mathcal{C}$ in place of $\mathcal{R}$ gives a well-defined $f$-divergence over $\mathcal{C}$. We formalize this with a definition.
\begin{definition}[Pre-Glivenko--Cantelli class]\label{def:preGC}
We say that $\mathcal{C}\subseteq \mathcal{B}$ is a pre-Glivenko--Cantelli class if for any Borel probability measure $\nu$,
\[
 G(\mathcal{C})\coloneqq \bigl\{g\in L^2(\nu) : \{ g>r \} \in \mathcal{C} \text{ for all } r\in\mathbb{R}\bigr\}
\]
is a closed convex cone in $L^2 (\nu)$; and for any pair of Borel probability measures $\mu$ and $\nu$,
\[
 \sup_{A\in\mathcal{C}}\, \bigl(\mu(A)-\nu(A)\bigr)=\int  \frac{1}{2}\Bigl\lvert \proj_{ G(\mathcal{C})}\frac{\d\mu}{\d\nu}-1\Bigr\rvert\,\d\nu.
\]
\end{definition}
We may speak freely of $\D^\mathcal{C}_f $ whenever we have a pre-Glivenko--Cantelli class $\mathcal{C}$ and an $f \in \mathcal{F}$.
We show that the first requirement in Definition~\ref{def:preGC} follows as long as $\mathcal{C}$ is closed under union and intersection, noting that both $\mathcal{R}$ and any $\sigma$-subalgebra of $\mathcal{B}$ will have this property.
\begin{lemma}\label{lem:moregen}
If $\mathcal{C} \subseteq \mathcal{B}$ is closed under countable unions and intersections,
then $ G(\mathcal{C})$ is a closed convex cone in $L^2(\nu)$.
\end{lemma}
\begin{proof}
The sum of two functions in $ G(\mathcal{C})$ remains in it as
\[
\{g_1+g_2 > c\} = \bigcup_{q\in\mathbb{Q}} \{g_1>c-q\}\cap \{g_2>q\} \in \mathcal{C}.
\]
From this it is routine to verify that $ G(\mathcal{C})$ is a convex cone. If $g_n\in G(\mathcal{C})$ is a sequence with $\lVert g_n-g \rVert \to 0$, then $g_n\to g$ in $\nu$-measure. By passing through a subsequence if necessary, we may assume that $g_n \to g$ almost surely. It then follows from the definition of pointwise convergence that
\[
\{g>c\} = \bigcup_{k\ge 1} \bigcap_{m\ge 1}\bigcup_{n\ge m} \Bigl\{g_n>c+\frac{1}{k}\Big\} \in \mathcal{C}
\]
and so $g \in  G(\mathcal{C})$. Hence $ G(\mathcal{C})$ is closed.
\end{proof}
The second requirement in Definition~\ref{def:preGC} is usually harder to establish. For  $\mathcal{B}$ it is immediate but for $\mathcal{R}$ it took nearly the whole of Section~\ref{sec:Rfdiv}. We  will extend the former to include all $\sigma$-subalgebras of $\mathcal{B}$.
\begin{lemma}
Any $\sigma$-subalgebra $\mathcal{C}\subseteq \mathcal{B}$ is a pre-Glivenko--Cantelli class. 
\end{lemma}
\begin{proof}
If $\mathcal{C}$ is a $\sigma$-subalgebra, then the restrictions of Borel probability measures $\mu$ and $\nu$ to $\mathcal{C}$, $\mu\vert_\mathcal{C}$ and $\nu\vert_\mathcal{C}$, remain probability measures \cite[p.~30]{Schilling2017}. Let $g$  be an almost surely nonnegative $\mathcal{C}$-measurable  function, i.e., $\{g>c\}\in\mathcal{C}$ for all $c\in\mathbb{R}$. Then 
\[
\int_A g\,\d\nu\vert_\mathcal{C} = \int_A g\,\d\nu
\]
for any $A\in\mathcal{C}$,  i.e., the value of the integral is preserved. If $\mu\ll\nu$, then $\mu\vert_\mathcal{C}\ll\nu\vert_\mathcal{C}$, ensuring that the Radon--Nikodym derivative $\d\mu\vert_\mathcal{C} / \d\nu\vert_\mathcal{C}$ exists and that
\[
\int_A \frac{\d\mu\vert_\mathcal{C}}{\d\nu\vert_\mathcal{C}} \,\d\nu =\int_A \frac{\d\mu\vert_\mathcal{C}}{\d\nu\vert_\mathcal{C}} \,\d\nu\vert_\mathcal{C} = \mu\vert_\mathcal{C}(A) = \mu(A) =\int_A \frac{\d\mu}{\d\nu} \,\d\nu
\]
for any $A\in\mathcal{C}$. So $\d\mu\vert_\mathcal{C} / \d\nu\vert_\mathcal{C}$ is indeed the conditional expectation and
\begin{equation}\label{eq:ceproj}
\frac{\d\mu\vert_\mathcal{C}}{\d\nu\vert_\mathcal{C}} = \mathbb{E}_{\mu}\Bigl[\frac{\d\mu}{\d\nu}\Bigm\vert\mathcal{C}\Bigr] =\proj_{G(\mathcal{C})}\frac{\d\mu}{\d\nu}
\end{equation}
by \cite[p.~90]{Bobrowski2005}. Hence
\begin{align*}
\D_{\tv}^\mathcal{C}(\mu \Vert \nu) &= \int  \frac{1}{2} \Bigl\lvert \proj_{G(\mathcal{C})}\frac{\d\mu}{\d\nu}-1\Bigr\rvert\,\d\nu =\int  \frac{1}{2}\Bigl\lvert \frac{\d \mu\vert_\mathcal{C}}{\d \nu\vert_\mathcal{C}}-1\Bigr\rvert\,\d\nu\\
& =\int  \frac{1}{2}\Bigl\lvert \frac{\d \mu\vert_\mathcal{C}}{\d \nu\vert_\mathcal{C}}-1\Bigr\rvert\,\d\nu\vert_\mathcal{C} =\sup_{A\in\mathcal{C}}\, \bigl(\mu(A)-\nu(A)\bigr). \qedhere
\end{align*}
\end{proof}
So Definition~\ref{def:preGC} is not vacuous: The class of rays $\mathcal{R}$ and any $\sigma$-subalgebra of $\mathcal{B}$ give two examples of  pre-Glivenko--Cantelli classes. The reason for introducing this notion is that any finite \textsc{vc}-dimensional pre-Glivenko--Cantelli class gives us a Vapnik--Chervonenkis theorem for $f$-divergence. 
\begin{theorem}[Vapnik--Chervonenkis theorem for $f$-divergence]\label{thm:postGC}
Let $f\in\mathcal{F}$ and $\mathcal{C}$ be a pre-Glivenko--Cantelli class. If $\dim_{\vc}(\mathcal{C}) < \infty$, then almost surely
\begin{equation}\label{eq:postGC}
\lim_{n\to\infty}\D_f^{\mathcal{C}}(\nu_n \Vert \nu)=0
=\lim_{n\to\infty}\D_f^{\mathcal{C}}(\nu \Vert \nu_n).
\end{equation}
\end{theorem}
\begin{proof}
All that we really need to observe is that the whole of Section~\ref{sec:GC} does not depend on the fact that $\mathcal{R}$ is the class of rays --- every result therein holds verbatim with $\mathcal{C}$ in place of $\mathcal{R}$ as long as $\D_f^\mathcal{C}$ is well-defined, i.e., as long as $\mathcal{C}$ is a pre-Glivenko--Cantelli class.  The original Vapnik--Chervonenkis theorem \cite{VapnikChervonenkis1971}, as applied in Propositions~\ref{prop:RN} and \ref{prop:RN2}, shows that a pre-Glivenko--Cantelli class $\mathcal{C}$ with finite \textsc{vc}-dimension must satisfy \eqref{eq:postGC} almost surely.
\end{proof}
With Theorem~\ref{thm:postGC} we are led to the following definition in analogy with standard Glivenko--Cantelli class (i.e., with respect to total variation distance).
\begin{definition}[Glivenko--Cantelli class for $f$-divergence]\label{def:postGC}
Let $f \in \mathcal{F}$ and  $\mathcal{C}\subseteq \mathcal{B}$ be a pre-Glivenko--Cantelli class  for $f$. We say that $\mathcal{C}$ is a Glivenko--Cantelli class for $f$ if  \eqref{eq:postGC} holds almost surely.
\end{definition}
Note that we need to restrict ourselves to pre-Glivenko--Cantelli classes so that $\D_f^{\mathcal{C}}$ is well-defined. It is nevertheless possible for such a Glivenko--Cantelli class to have infinite \textsc{vc}-dimension, so Theorem~\ref{thm:postGC} merely provides a sufficient criterion.
\begin{example}[Infinite \textsc{vc}-dimensional Glivenko--Cantelli class for $f$-divergence] 
Let $\mathcal{E}=(E_n)_{n\in\mathbb{N}}$ be a countable partition of $\mathbb{R}$ with measurable $E_n$, $n \in \mathbb{N}$. Let $\mathcal{C}=\sigma(\mathcal{E})\subseteq\mathcal{B}$ be the $\sigma$-algebra generated by $\mathcal{E}$. So any set in $\mathcal{C}$ can be expressed as a countable union of sets in  $\mathcal{E}$ \cite[p.~4]{Cinlar2011}. Thus any $\mathcal{C}$-measurable function must be a simple function $h=\sum_{n=1}^{\infty} \alpha_n\1_{E_n}$.
Any pair of Borel probability measures $\mu$ and $\nu$ restrict to probability measures $\mu\vert_\mathcal{C}$ and $\nu\vert_\mathcal{C}$ with Radon--Nikodym derivative
\[
\frac{\d\mu\vert_\mathcal{C}}{\d\nu\vert_\mathcal{C}}=\sum_{n=1}^{\infty} \alpha_n\1_{E_n}.
\]
For any $m \in \mathbb{N}$,
\[
\mu\vert_\mathcal{C}(E_m) = \int_{E_m} \frac{\d \mu\vert_\mathcal{C}}{\d \nu\vert_\mathcal{C}} \,\d\nu\vert_\mathcal{C}	 =\int  \1_{E_m} \sum_{n=1}^{\infty} \alpha_n\1_{E_n} \,\d\nu\vert_\mathcal{C} =\alpha_m \cdot\nu\vert_\mathcal{C}(E_m).
\]
By \eqref{eq:ceproj},
\[
\proj_{G(\mathcal{C})} \frac{\d\mu}{\d\nu}
=\frac{\d\mu\vert_\mathcal{C}}{\d\nu\vert_\mathcal{C}}=\sum_{n=1}^{\infty}\1_{E_n} \frac{\mu\vert_\mathcal{C}(E_n)}{\nu\vert_\mathcal{C}(E_n)},
\]
and so the $f$-divergence over $\mathcal{C}$ is given by
\[
\D_f^\mathcal{C}(\nu_n \Vert \nu) =  \D_f(\nu_n\vert_\mathcal{C} \Vert \nu\vert_\mathcal{C})
=\sum_{i=1}^{\infty} f\Bigl(\frac{\nu_n\vert_\mathcal{C}(E_i)}{\nu\vert_\mathcal{C}(E_i)}\Bigr) \nu\vert_\mathcal{C}(E_i)	.
\]
If the sequence $\bigl(\nu_n\vert_\mathcal{C}(E_k) / \nu\vert_\mathcal{C}(E_k)\bigr)_{k \in \mathbb{N}}$ is uniformly bounded, then
\begin{align*}
	\limsup_{n\to\infty}\D_f^\mathcal{C}(\nu_n \Vert \nu) &= \limsup_{n\to\infty}\sum_{i=1}^{\infty} f\Bigl(\frac{\nu_n\vert_\mathcal{C}(E_i)}{\nu\vert_\mathcal{C}(E_i)}\Bigr) \nu\vert_\mathcal{C}(E_i)\\
	&\le \sum_{i=1}^{\infty} \limsup_{n\to\infty} f\Bigl(\frac{\nu_n\vert_\mathcal{C}(E_i)}{\nu\vert_\mathcal{C}(E_i)}\Bigr) \nu\vert_\mathcal{C}(E_i) =\sum_{i=1}^{\infty} f\Bigl(\lim_{n\to\infty} \frac{\nu_n\vert_\mathcal{C}(E_i)}{\nu\vert_\mathcal{C}(E_i)}\Bigr) \nu\vert_\mathcal{C}(E_i) =0.
\end{align*}
So  $\lim_{n\to\infty} \D_f^\mathcal{C}(\nu_n \Vert \nu) = 0$ almost surely. Likewise $\lim_{n\to\infty} \D_f^\mathcal{C}(\nu \Vert \nu_n) = 0$ almost surely.
Hence $\mathcal{C}$ satisfies \eqref{eq:postGC} almost surely even though $\dim_{\vc}(\mathcal{C}) = \infty$.
\end{example}

\section{Choquet integral}\label{sec:cho}

To address Question~\ref{ques:1}, our initial thought was to employ the Choquet integral \cite{choquet1953, Denneberg1994}, promising because it allows one to work with an arbitrary class $\mathcal{C} \subseteq 2^\Omega$ with essentially no condition imposed (unlike $\pi$-system or $\sigma$-algebra). This is perfect for us as it seems that we could have taken $\mathcal{C} = \mathcal{R}$ but it did not work. We think it might be instructive to record the difficulties encountered as we remain optimistic that it is not a dead end.

The Lebesgue integral defines integration with respect to a measure; Choquet integral extends this to  so-called ``capacities.'' Let $\Omega$ be a set and $\mathcal{C}\subseteq 2^\Omega$ be a collection of subsets with $\varnothing,\Omega\in\mathcal{C}$. A function $\nu:\mathcal{C}\to\mathbb{R}$ is called a \emph{capacity} on $\mathcal{C}$ if it satisfies
\begin{enumerate}[\upshape (a)]
\item normalization: $\nu(\varnothing)=0$;
\item monotonicity: $\nu(A)\le \nu(B)$ whenever $A\subseteq B$,  $A,B\in\mathcal{C}$.
\end{enumerate}
A function $g:\Omega\to\mathbb{R}$ is $\mathcal{C}$-measurable if $\{g>t\} \in \mathcal{C}$ for all $t\in\mathbb{R}$. Let $B_\p (\mathcal{C})$ be the set of all bounded, nonnegative, $\mathcal{C}$-measurable functions. The distribution of $g\in B_\p (\mathcal{C})$ is the function $m_{\nu, g} : \mathbb{R} \to [0,\infty)$, $t \mapsto \nu(\{ g>t \})$. The Choquet integral of the function $g$ with respect to the capacity $\nu$ is then
\[
	\Cint  g\,\d\nu \coloneqq \int_0^\infty m_{\nu,g}(t)\,\d t,
\]
where the right-hand side is a Riemann integral.  Choquet coincides with Lebesgue when $\nu$ is a measure \cite[p.~62]{Denneberg1994}. 

For the class of rays $\mathcal{R}$, it is easy to see that $\nu_\mathcal{R}$, the restriction of $\nu$ to $\mathcal{R}$, is a capacity. So we could simply define $f$-divergence over $\mathcal{R}$ as
\[
\D_f^\mathcal{R}(\mu \Vert \nu) \coloneqq \D_f(\mu_\mathcal{R} \Vert \nu_\mathcal{R})  = \Cint f\Bigl(\frac{\d\mu_\mathcal{R}}{\d\nu_\mathcal{R}}\Bigr)\,\d\nu_\mathcal{R}.
\]
Two immediate problems are that the Radon--Nikodym derivative $\d\mu_\mathcal{R} / \d\nu_\mathcal{R}$ cannot be easily defined and the function $f(\d\mu_\mathcal{R} / \d\nu_\mathcal{R})$ may not be $\mathcal{R}$-measurable, i.e., nonincreasing. The first problem has been discussed extensively \cite{Graf1980, Harding1997, Narukawa2007, Sugeno2013, Torra2016} but all assumed that $\mathcal{C}$ is a $\sigma$-algebra and thus inapplicable to $\mathcal{R}$. One might think that a way around these problems is to extend the domain of $\mu_\mathcal{R}$ and $\nu_\mathcal{R}$ back to the whole of $\mathcal{B}$, i.e.,
\[
\widehat\nu_\mathcal{R}(A) \coloneqq
\begin{cases}
\nu(A) & \text{if } A \in \mathcal{R}, \\
0 & \text{if } A \in \mathcal{B} \setminus \mathcal{R},
\end{cases}
\]
and likewise for $\widehat\mu_\mathcal{R}$. But now $\widehat\mu_\mathcal{R}$ and $\widehat\nu_\mathcal{R}$ fail monotonicity and are not even capacities.

We are slightly disappointed that the discussions in \cite{Graf1980, Harding1997, Narukawa2007, Sugeno2013, Torra2016} are all based on the assumption that  $\mathcal{C}$ is a $\sigma$-algebra (if so, then why not just use standard Lebesgue integral?)\ but interested readers may have better luck where we and these authors failed:
\begin{open}
Is it possible to define an $f$-divergence using Choquet integral for a reasonably general class $\mathcal{C} \subseteq \mathcal{B}$ so that Theorems~\ref{thm:prop} and \ref{thm:rel} hold true?
\end{open}
By ``reasonably general'' we expect $\mathcal{C}$ to cover at least the requirements of Lemma~\ref{lem:moregen}, or perhaps be a $\pi$-system or a $\lambda$-system, certainly not just a $\sigma$-algebra. A positive answer to this open problem should provide an alternative to Definition~\ref{def:Rfdiv} and could potentially be useful for developing the Vapnik--Chervonenkis theory for $f$-divergence in Section~\ref{sec:VC}.

\section{Conclusion}

As we have noted in Section~\ref{sec:intro}, the total variation distance plays an outsize role in traditional theoretical statistics but modern AI applications often relies on $f$-divergences other than the total variation distance. The work in this article would hopefully shed light on how one may extend classical results in statistics to other $f$-divergences.

\bibliographystyle{abbrv}

\end{document}